\documentclass[12pt,a4paper]{article}

\usepackage[T1]{fontenc}%
\usepackage[utf8]{inputenc}%
\usepackage{lmodern}
\usepackage{pdfpages}

\usepackage[german,english]{babel} 
\usepackage{datetime}
\usepackage{advdate}
\usepackage{lastpage}%
\usepackage{fancyhdr}%
\usepackage{setspace}%
\usepackage{geometry}%
\usepackage{relsize}%
\usepackage{enumitem}
\usepackage{csquotes}
\usepackage[pdfusetitle]{hyperref}
\usepackage{caption}
\usepackage{float}

\usepackage{amsmath}%
\usepackage{amssymb}%
\usepackage[hyperref,thmmarks,amsmath]{ntheorem}
\usepackage{mathtools}%
\usepackage{units}%
\usepackage{faktor}%
\usepackage[makeroom]{cancel}
\usepackage{calrsfs} 
\usepackage{bbm}

\usepackage{tikz}
\usetikzlibrary {arrows.meta}
\usetikzlibrary{decorations.markings}
\usetikzlibrary{decorations.pathreplacing,calligraphy}
\usetikzlibrary{patterns}

\usepackage{subcaption}

\usepackage{cleveref}

\DeclareMathAlphabet{\mathpcal}{OMS}{pzc}{m}{n}

\newcommand\R{\mathbb R}

\newcommand\Z{\mathbb Z}

\newcommand\T{\mathbb T}
\newcommand\RP{\mathbb {R}\textnormal{P}}

\DeclareMathOperator{\sys}{sys}

\DeclareMathOperator{\Diff}{Diff}
\DeclareMathOperator{\Isom}{Isom}
\DeclareMathOperator{\area}{area}
\DeclareMathOperator{\Var}{Var}

\newcommand\restr[2]{{
		\left.\kern-\nulldelimiterspace 
		#1
		\vphantom{|} 
		\right|_{#2} 
}}

\newtheorem{ThmIntro}{Theorem}
\newtheorem{Def}{Definition}[section] 

\newtheorem{Thm}[Def]{Theorem}	
\newtheorem{Lemma}[Def]{Lemma}
\newtheorem{Cor}[Def]{Corollary}
\newtheorem{Rem}[Def]{Remark}

\theoremstyle{plain}
\theoremseparator{.}

\theoremstyle{plain}
\theoremseparator{:}
\theorembodyfont{\normalfont}
\newtheorem{Case}{Case} 

\theoremstyle{plain}
\theoremsymbol{\ensuremath{\Box}}

\theoremstyle{nonumberplain}
\theoremsymbol{\ensuremath{\Box}}
\theoremseparator{.}
\theoremheaderfont{\normalfont\itshape}
\theorembodyfont{\normalfont}

\theoremstyle{empty}
\theoremsymbol{\ensuremath{\bigtriangleup}}
\newtheorem{proof}{}

\Crefname{equation}{Equation}{Equations}  

\Crefname{section}{Section}{Sections}
\Crefname{Claim}{Claim}{Claims}
\Crefname{Rem}{Remark}{Remarks}
\Crefname{Cor}{Corollary}{Corollaries}
\Crefname{Lemma}{Lemma}{Lemmas}
\Crefname{Thm}{Theorem}{Theorems}
\Crefname{ThmIntro}{Theorem}{Theorems}
\Crefname{Prop}{Proposition}{Propositions}
\Crefname{Def}{Definition}{Definitions}
\Crefname{enumi}{}{}
\Crefname{Case}{Case}{Cases}

\Crefname{figure}{Fig.}{Figs.}
\usepackage[backend=biber, style=alphabetic, maxnames=15, firstinits=true]{biblatex}
\renewbibmacro{in:}{} 

\AtEveryBibitem{%
	\clearfield{day}%
	\clearfield{month}%
	\clearfield{endday}%
	\clearfield{endmonth}%
}

\title{Stability of systolic inequalities for the Möbius strip and Klein bottle}
\author{Jan Eyll\thanks{Ruhr-Universität Bochum, Fakultät für Mathematik, Bochum, Germany\\ jan.eyll@rub.de}}
\date{}

\begin{document}
	\maketitle
	\begin{abstract}
		The systolic area $\alpha_{sys}$ of a nonsimply connected compact Riemannian surface $(M,g)$ is defined as its area divided by the square of the systole, where the systole is equal to the length of a shortest noncontractible closed curve. The systolic inequality due to Bavard states that on the Klein bottle, the systolic area has the optimal lower bound $\frac{2\sqrt{2}}{\pi}$. Bavard also constructed metrics of minimal systolic area in any given conformal class. 
		We give an alternative proof of these results, which also yields an estimate on the systolic defect $\alpha_{sys}-\frac{2\sqrt{2}}{\pi}$ in terms of the $L^2$-distance of the conformal factor to the metric which minimizes the systolic area. On the Möbius strip, we also prove similar estimates for metrics in fixed conformal classes.
	\end{abstract}
	\section*{Introduction}
	The \textit{systole} of a nonsimply connected Riemannian manifold $(M,g)$ is defined as the infimum of the lengths of noncontractible closed curves, denoted by $\sys(M,g)$. If $M$ is compact, then the systole is positive and there exists a noncontractible closed curve of length $\sys(M,g)$, which is necessarily a closed geodesic. If $M$ is two-dimensional, we denote the Riemannian area induced by $g$ with $\area(M,g)$. Then the \textit{systolic area} of $(M,g)$ is defined as 
	\begin{equation*}
		\alpha_{sys}(M,g):=\frac{\area(M,g)}{\sys^2(M,g)}.
	\end{equation*}
	Clearly, $\alpha_{sys}(M,g)$ is invariant under rescalings of the metric. For a set $C$ of metrics on $M$, we define 
	\begin{equation}
		\alpha_{sys}(M,C):=\inf_{g\in C} \alpha_{sys}(M,g).
	\end{equation}
	If $C$ is the set of all metrics on $M$, we simply write $\alpha_{sys}(M)=\alpha_{sys}(M,C)$. 
	
	In 1949, Loewner discovered that on the two-dimensional torus $\T^2$, the systolic area has a positive lower bound. More precisely, he proved 
	\begin{equation}
		\alpha_{sys}(\T^2)=\frac{\sqrt{3}}{2}
	\end{equation}
	and showed that this lower bound for $\alpha_{sys}(\T^2,g)$ is attained if and only if $(\T^2,g)$ is isometric to a flat equilateral torus, which is defined as the quotient of the standard Euclidean plane by a hexagonal lattice. Similar results also exist for the real projective plane $\RP^2$ (\cite{Pu1952}) and the Klein bottle $K$ (\cite{Bavard1986}), namely
	\begin{align}
		\alpha_{sys}(\RP^2)&= \frac{2}{\pi},\\
		\alpha_{sys}(K)&=\frac{2\sqrt{2}}{\pi}. \label{Eq: Bavards original ineq}
	\end{align}
	For the projective plane, the lower bound is attained uniquely (up to isometry and rescaling) by the round metric. 
	In contrast, for the Klein bottle, the lower bound is attained only by continuous, but not smooth metrics, as we will see below.
	It is also worth noting that these three examples are the only surfaces for which the value of $\alpha_{sys}(M)$  is known. 
	$\alpha_{sys}$ is known to be bounded from below by $1/2$ (\cite{Hebda1982,Gromov1983}) and by a bound that grows with the genus of $M$ (\cite{Gromov1983}), but these bounds are not sharp.
	
	In \cite{Pu1952}, Pu also established systolic inequalities for every conformal class of Riemannian metrics on the M\"obius strip, which we discuss in \Cref{Sec: Mobius inequality}. More general inequalities on the M\"obius strip can be found in \cite{Blatter1961a}, where Blatter proves inequalities relating the area to the product of the minimal lengths of curves in specific homotopy classes. Furthermore, \cite{ElMir2015} contains similar inequalities for every conformal class of Riemannian metrics on the Klein bottle. We will see that the homotopy classes considered there are relevant for our discussion of the systole on the Klein bottle, as they contain natural candidates for curves attaining the systole. 
	
	More recently, these results have been generalized to Finsler metrics (see \cite{Balacheff2024} for $\T^2$, \cite{Ivanov2011} for the projective plane and \cite{Sabourau2016} for the Klein bottle) and questions concerning the stability of the implied systolic inequalities have been studied: 
	The difference 
	\begin{equation*}
		\alpha_{sys}(M,g)-\alpha_{sys}(M)
	\end{equation*}
	is called the \textit{systolic defect}, and the goal is to understand whether a Riemannian metric $g$ with small systolic defect is necessarily close to an optimal metric $\hat{g}$, i.e. a metric for which $\alpha_{sys}(M,\hat{g})=\alpha_{sys}(M)$. In this situation, one cannot expect $g$ and $\hat{g}$ to be close in the $C^0$-sense, because modifying $\hat{g}$ by adding \enquote{spikes} of small area has little effect on the systolic area but can make the $C^0$-distance arbitrarily large. It turns out that using $L^2$-distance of conformal factors gives good results, due to its close relations to the Riemannian area.  
	
	Due to the conformal representation theorem, any metric $g$ on a closed surface is of the form $g=\phi^2g_0$, where $\phi$ is a positive function and $g_0$ a metric of constant curvature. 
	In some cases, this can be used to estimate the systolic defect: let $g_{round}$ be the metric of constant curvature $1$ on $\RP^2$ and $\phi^2g_{round}$ any metric. It is shown in \cite{Katz2020} that 
	\begin{equation*}
		\frac{2\pi \Var(\phi)}{\sys^2(\RP^2,\phi^2g_{round})}\leq \alpha_{sys}(\RP^2,\phi^2g_{round})-\alpha_{sys}(\RP^2),
	\end{equation*}
	where the variance $\Var(\phi)$ is defined as follows: 
	denote by $\mu$ the probability measure on $\RP^2$ obtained by normalizing the Riemannian area measure induced by $g_{round}$. 
	Then $E(\phi):=\int_{\RP^2} \phi\, d\mu$ is the expected value and 
	\begin{equation}
		\Var(\phi):=\int_{\RP^2} \left(\phi - E(\phi)\right)^2 d\mu = E(\phi^2)-E(\phi)^2
	\end{equation}
	is the variance of $\phi$. Because the variance measures the $L^2$-distance between $\phi$ and its mean $E(\phi)$, these estimates show that metrics with small systolic defect are $L^2$-close to metrics with constant conformal factor, (i.e. round metrics), which are precisely the optimal metrics on $\RP^2$. 
	A very similar result for the torus can be found in \cite{Horowitz2009}. 
	
	Our goal is to prove an analogous result for the Klein bottle. Here, the main difference to the projective plane and the torus is that on the Klein bottle, the metric attaining the lower bound is not of constant curvature (and consequently the conformal factor $\phi$ is not constant in this case), so one cannot expect the variance of the conformal factor as a lower bound for $\alpha_{sys}(K,\phi^2g_{flat})-\alpha_{sys}(K)$. 
	However, we will see that by changing the considered probability measure, our result can also be expressed in terms of the variance of a conformal factor (see \Cref{Rem: Klein bottle variance} below). 
	
	The space $\mathpcal{C}(K)$ of conformal classes of metrics on the Klein bottle is one-dimensional and can be parametrized by the interval $(0,+\infty)$. 
	For more details on this fact, see \Cref{Lemma: Characterization of CK}, where an explicit parametrization $(0,+\infty)\to \mathpcal C(K)$, $\beta\mapsto C_\beta$ is given.
	
	In \cite{Bavard1988}, Bavard showed that for every conformal class $C\in \mathpcal C(K)$, there exists a continuous, not necessarily smooth metric $g_C\in C$ such that 
	\begin{equation}\label{Eq: Bavards conformal equation}
		\alpha_{sys}(K,C)=\alpha_{sys}(K,g_C).
	\end{equation}
	He also proved that the metric $g_C$ with this property is unique up to rescaling and isometry. 
	An explicit description of $g_C$ can be found in \Cref{Sec: Klein bottle inequality}. 
	We will see that it is in general not of constant curvature, which is related to the fact that the isometry group of a constant curvature metric on the Klein bottle (i.e. a flat metric) does not act transitively on $K$. This makes the study of conformal systolic inequalities on $K$ more difficult than on $\RP^2$ and $\T^2$. 
	
	As a result of Bavard's work, the function $C\mapsto \alpha_{sys}(K,C)$ can be computed explicitly. It has a unique minimizer $C_0\in \mathpcal C(K)$ with the property 
	\begin{equation}
		\alpha_{sys}(K,C_0)=\inf_{C\in \mathpcal C(K)} \alpha_{sys}(K,C)=\alpha_{sys}(K)=\frac{2\sqrt{2}}{\pi}.
	\end{equation}
	The graph of this function after the identification of $\mathpcal C(K)$ with $(0,+\infty)$ is sketched in \Cref{Fig: Graph of C mapsto alpha_sys}.

	\begin{figure}
		\centering
		\begin{tikzpicture}[domain=-.1:1.8,scale=3.5]

			\draw[->] (-.1,0) -- (1.8,0) node[right] {$\beta$};
			\draw[->] (0,-.3) -- (0,1.6);
			
			\draw[dashed, thin] (-.1,0.9003)--(1.8,0.9003);
			\node[anchor=east] at  (-.1,1){$\frac{2\sqrt{2}}{\pi}$};	
			
			\draw[dashed, thin] (1.5707/2,-.1) -- (1.5707/2,1.414);
			\node[anchor=north east] at  (1.65/2,-.1){$\pi/4$};	
			
			\draw[dashed, thin] (1.762/2,-.1) -- (1.762/2,1.414);
			\node[anchor=north west] at  (1.7/2,-.1){$\beta_0$};	
			
			\draw[dashed, thin] (2.634/2,-.1) -- (2.634/2,1.414);	
			\node[anchor=north] at  (2.634/2,-.1){$\beta_1$};
			
			\node[anchor=south west] at (1.7,1.3) {$\alpha_{sys}(K,C_\beta)$};
			
			\draw[thick,domain = .4:1.5707/2]    plot (\x,{3.14159/(4*\x)}  ) ;

			\draw[ thick] plot coordinates {
				(0.88137,0.9003180995712282) (0.8813666968366763,0.900320720883254) (0.8813581167695785,0.9003262264617308) (0.8813461349411306,0.9003339385638889) (0.881330767519656,0.9003438603914445) (0.881312030673899,0.9003559950792501) (0.8812899405731268,0.9003703456946514) (0.8812645133872351,0.9003869152368437) (0.8812357652868522,0.9004057066362279) (0.881203712443443,0.9004267227537635) (0.8811683710294116,0.9004499663803215) (0.881129757218206,0.9004754402360361) (0.8811292230576441,0.9004757650883733) (0.881087887184421,0.9005031469696555) (0.8810427771039022,0.9005330891578913) (0.8809944431538475,0.9005652693047663) (0.8809429015129114,0.9005996898409635) (0.880888446115288,0.9006361554354008) (0.8808881683613072,0.9006363531231714) (0.8808302598809088,0.9006752614334298) (0.8807691922553539,0.900716416978476) (0.8807049816701463,0.9007598218890881) (0.8806476691729324,0.9007983865933548) (0.8806376443127564,0.9008054782194277) (0.8805671963727243,0.9008533879463839) (0.8804936540417608,0.9009035529689148) (0.880417033513849,0.900955975107389) (0.8804068922305764,0.9009626746717262) (0.8803373509853452,0.9010106561029266) (0.8802546226550798,0.9010675976167396) (0.8801688647244581,0.9011268012294714) (0.8801661152882205,0.9011286344078432) (0.8800800933975611,0.9011882684405373) (0.879988324881246,0.9012520006674632) (0.8799253383458646,0.9012953471049888) (0.8798935753852456,0.9013179992452242) (0.8797958611222689,0.9013862654255832) (0.8796951983081007,0.9014568003764303) (0.8796845614035087,0.9014640570083164) (0.8795916031617014,0.90152960518112) (0.8794850919053054,0.9016046808378096) (0.8794437844611529,0.90163332232412) (0.8793756807645218,0.9016820282587977) (0.879263385968432,0.901761648269862) (0.879203007518797,0.9018039870605071) (0.8791482237496883,0.9018435416095972) (0.879030210344614,0.9019277089287538) (0.8789622305764411,0.9019757432764932) (0.8789093619932999,0.9020141507895761) (0.8787856949397037,0.9021028676651413) (0.8787214536340852,0.9021484861033843) (0.8786592254317463,0.9021938599386975) (0.8785299697214116,0.9022871279030045) (0.8784806766917294,0.902322235413179) (0.8783979440648418,0.9023826717596716) (0.878263164722436,0.9024804916184993) (0.8782398997493734,0.9024970858750618) (0.8781256479589465,0.9025805874968186) (0.8779991228070175,0.902672735093866) (0.877985410043576,0.902682959318833) (0.877842467250073,0.9027876069149596) (0.8777583458646616,0.9028487763824001) (0.87769683585683,0.9028945300211724) (0.877548532146978,0.9030037282783445) (0.8775175689223057,0.9030261959284389) (0.8773975724084828,0.9031152012315923) (0.8772767919799499,0.9032042346901703) (0.8772439729342406,0.9032289483296211) (0.8770877500221739,0.9033449689240697) (0.877036015037594,0.903382964317878) (0.8769289199753258,0.9034632622688581) (0.8767952380952381,0.9035627612903306) (0.8767674991019547,0.9035838275195344) (0.8766035037156304,0.9037066637326241) (0.8765544611528822,0.9037429939677649) (0.8764369501353275,0.9038317698649805) (0.8763136842105264,0.9039241386809456) (0.876267854685519,0.9039591447731353) (0.8760962336962717,0.9040887872126517) (0.8760729072681704,0.9041061736176322) (0.8759221035033383,0.9042206958374774) (0.8758321303258146,0.9042884319527092) (0.8757454804482528,0.9043548691993016) (0.8755913533834586,0.9044719872060426) (0.875566380878422,0.9044913057469105) (0.8753848211472186,0.9046300038255458) (0.8753505764411027,0.904655865757704) (0.8752008176140749,0.9047709616762654) (0.8751097994987469,0.9048402678577437) (0.8750143866445742,0.9049141774353047) (0.8748690225563909,0.9050257207785674) (0.8748255446105445,0.9050596491334403) (0.874634307890148,0.9052073746953547) (0.8746282456140351,0.9052119961780959) (0.8744406928679749,0.9053573519390046) (0.8743874686716792,0.9053982082673625) (0.8742447159351344,0.9055095785749895) (0.8741466917293232,0.9055853152621399) (0.8740463934893453,0.9056640522059224) (0.8739059147869674,0.9057732561111628) (0.8738457419350274,0.9058207703258044) (0.8736651378446115,0.9059619767689305) (0.8736427776833913,0.9059797303193987) (0.8734375171525304,0.90614092946161) (0.8734243609022556,0.906151144083015) (0.8732299767675091,0.9063043649168627) (0.8731835839598997,0.9063405754642557) (0.8730201729604546,0.9064700337384856) (0.8729428070175438,0.9065307219877722) (0.8728081221706444,0.9066379328680942) (0.872702030075188,0.9067215495476026) (0.8725938408445978,0.9068080591349814) (0.872461253132832,0.9069130281222211) (0.872377345436163,0.9069804092555038) (0.8722204761904762,0.9071051313728357) (0.8721586524066066,0.9071549798324776) (0.8719796992481202,0.907297836280568) (0.8719377782247035,0.9073317673545733) (0.8717389223057644,0.9074911228185012) (0.8717147393668222,0.9075107681957126) (0.8714981453634085,0.9076849736550436) (0.8714895523170154,0.9076919786144716) (0.8712622335671061,0.9078753947534836) (0.8712573684210526,0.9078792800440841) (0.8710327996167757,0.9080610126388473) (0.8710165914786967,0.9080740013722762) (0.8708012669736509,0.908248828179537) (0.8707758145363409,0.9082692927399934) (0.8705676521533908,0.908438837166816) (0.870535037593985,0.90846514357214) (0.8703319716797724,0.9086310352736535) (0.8702942606516291,0.9086615448976123) (0.8700942420847785,0.908825418054145) (0.8700534837092732,0.9088584891874645) (0.869854479908682,0.9090219809429362) (0.8698127067669172,0.9090559702076363) (0.8696127017001326,0.9092207192546499) (0.8695719298245614,0.9092539828848858) (0.8693689240162421,0.909421628183317) (0.8693311528822055,0.9094525231846785) (0.8691231634226687,0.9096247028018106) (0.8690903759398496,0.9096515879999201) (0.868875436493703,0.9098299380612851) (0.8688495989974937,0.9098511750495366) (0.8686257598123517,0.9100373287906174) (0.8686088220551378,0.9100512827859892) (0.8683741499704214,0.910246869695853) (0.868368045112782,0.9102519103109104) (0.868127268170426,0.9104529456185719) (0.8681206235686042,0.9104585553596568) (0.8678864912280702,0.9106543698592859) (0.8678651972165589,0.9106723802407657) (0.8676457142857144,0.910856288590101) (0.8676078875329961,0.9108883386734479) (0.8674049373433583,0.9110587039977325) (0.8673487111457613,0.9111064248669654) (0.8671641604010025,0.9112616185634118) (0.8670876846919154,0.9113266329050388) (0.8669233834586466,0.9114650350145357) (0.8668248248178203,0.9115489567453215) (0.8666826065162906,0.9116689562805952) (0.8665601481792182,0.9117733902188725) (0.8664418295739348,0.9118733854530114) (0.866293671441315,0.911999927029638) (0.866201052631579,0.9120783257485259) (0.8660254112788609,0.9122285607539355) (0.8659602756892231,0.9122837804758391) (0.8657553843762323,0.9124592848399438) (0.8657194987468672,0.912489753005203) (0.8654836074275117,0.9126920926071961) (0.8654787218045112,0.9126962467407171) (0.8652379448621554,0.912902845388905) (0.8652100971365693,0.912926977246081) (0.8649971679197995,0.913109882849332) (0.8649348702171421,0.9131639318173447) (0.8647563909774436,0.9133174378510055) (0.8646579433929151,0.9134029492516011) (0.8645156140350878,0.9135255145929702) (0.8643793333975992,0.9136440223488465) (0.8642748370927318,0.9137341171778839) (0.8640990569750121,0.9138871437779773) (0.864034060150376,0.913943249597091) (0.8638171308791563,0.9141323060763168) (0.8637932832080201,0.914152915717114) (0.8635525062656642,0.914362849717122) (0.8635335718742977,0.9143795016491435) (0.8633117293233082,0.9145729722562793) (0.8632483967350443,0.9146287227692284) (0.8630709523809523,0.9147836308457157) (0.862961622246424,0.9148799615763745) (0.8628301754385965,0.9149948294940873) (0.8626732652039617,0.9151332100769647) (0.8625893984962406,0.9152065720146044) (0.8623833424137578,0.9153884601435134) (0.8623486215538847,0.9154188620185655) (0.8621078446115288,0.9156314860353574) (0.8620918706925641,0.9156457035142256) (0.8618670676691729,0.9158441823636523) (0.8617988668678609,0.91590493179256) (0.8616262907268171,0.9160574314125813) (0.861504347777933,0.9161661364467999) (0.8613855137844612,0.9162712365866189) (0.8612083302719462,0.9164293088096286) (0.8611447368421052,0.9164856010486917) (0.8609108312100225,0.9166944400777124) (0.8609039598997493,0.9167005277178758) (0.8606631829573934,0.9169153557070496) (0.8606118674633155,0.9169615213112878) (0.8604224060150375,0.9171306598015272) (0.8603114559140855,0.9172305434337581) (0.8601816290726817,0.9173465326366465) (0.8600096134557743,0.9175014972312923) (0.8599408521303258,0.9175629767060973) (0.8597063569930785,0.9177743733524342) (0.8597000751879699,0.9177799942394762) (0.8594592982456141,0.9179968705093574) (0.8594017034420252,0.918049162307716) (0.8592185213032582,0.9182142462397708) (0.8590956697300435,0.918325854469279) (0.8589777443609022,0.9184322024683899) (0.8587882727960402,0.9186044400705012) (0.8587369674185463,0.9186507408761662) (0.8584961904761904,0.9188696612023208) (0.8584795295904699,0.918884909205631) (0.8582554135338346,0.9190885413190972) (0.8581694570754108,0.9191672518294294) (0.8580146365914787,0.9193080109188732) (0.8578580722246335,0.9194514577568162) (0.8577738596491228,0.9195280710588333) (0.8575453920236752,0.9197375166625272) (0.8575330827067669,0.919748722518355) (0.8572923057644111,0.9199692580858927) (0.8572314334699116,0.920025418080776) (0.8570515288220552,0.9201902466154186) (0.8569162135726254,0.9203151514049217) (0.8568107518796992,0.9204118333155034) (0.8565997493530799,0.9206067058871475) (0.8565699749373433,0.9206340182893844) (0.8563291979949874,0.920856269109929) (0.8562820578445879,0.9209000706381443) (0.8560884210526317,0.9210787847736912) (0.8559631560925829,0.9211952346268024) (0.8558476441102757,0.9213019052187701) (0.8556430611546876,0.9214921866799102) (0.8556068671679198,0.9215256298409854) (0.8553660902255639,0.9217494723718449) (0.8553217901007842,0.9217909154818615) (0.855125313283208,0.9219735246501921) (0.8549993600130835,0.9220914095743697) (0.8548845363408522,0.9221981870820368) (0.8546757879861929,0.9223936573561899) (0.8546437593984962,0.9224234583375113) (0.8544029824561404,0.92264878550901) (0.8543510911271855,0.9226976470828496) (0.8541622055137844,0.9228743840582931) (0.8540252865556671,0.9230033668663857) (0.8539214285714285,0.9231005967317845) (0.8536983914038452,0.9233108046750925) (0.8536806516290727,0.9233274214656702) (0.8534398746867169,0.9235541410508339) (0.853370422816595,0.9236199483332741) (0.8531990977443609,0.9237812948248226) (0.8530413979515269,0.9239307855210079) (0.852958320802005,0.9240090651743129) (0.8527175438596492,0.9242373866251796) (0.8527113339790526,0.9242433037739147) (0.8524767669172932,0.9244654824768572) (0.8523802480824519,0.9245574904829384) (0.8522359899749373,0.9246941990298398) (0.8520481574579376,0.9248733328941332) (0.8519952130325814,0.924923533008306) (0.8517544360902257,0.9251530966458459) (0.8517150793147216,0.9251908181084594) (0.8515136591478697,0.9253827609611243) (0.8513810308750803,0.9255099330815894) (0.8512728822055138,0.9256130459005659) (0.8510460293744178,0.9258306646237187) (0.8510321052631579,0.9258439474612112) (0.850791328320802,0.9260746949783937) (0.8507100920613314,0.92615299939939) (0.8505505513784462,0.9263059326907885) (0.8503732361976755,0.9264769239273237) (0.8503097744360902,0.9265377892534281) (0.8500689974937343,0.9267699504892636) (0.8500354790586245,0.9268024245802567) (0.8498282205513784,0.9270021433025923) (0.8496968379327369,0.9271294875847929) (0.8495874436090225,0.9272349566667202) (0.849357330122017,0.9274580990212595) (0.8493466666666667,0.9274683853968507) (0.8491058897243108,0.9277016322799058) (0.8490169729419785,0.9277882448235755) (0.8488651127819549,0.9279354045711405) (0.8486757837217054,0.9281199107791265) (0.848624335839599,0.9281697929100874) (0.848383558897243,0.9284043582569903) (0.8483337798039146,0.9284530825286519) (0.8481427819548872,0.9286390917918166) (0.8479909785450166,0.9287877455661377) (0.8479020050125313,0.928874441485591) (0.8476612280701754,0.92911028348278) (0.847647397315177,0.9291238852387238) (0.8474204511278195,0.9293459806516359) (0.8473030534983764,0.929461486746615) (0.8471796741854637,0.9295822935064412) (0.8469579644924711,0.9298005351430063) (0.8469388972431078,0.929819215275616) (0.8466981203007519,0.9300560365610149) (0.8466121477092523,0.9301410153340153) (0.846457343358396,0.9302933143475471) (0.8462656205745062,0.930482912078626) (0.84621656641604,0.9305311995446914) (0.8459757894736842,0.9307692276502175) (0.8459184005280727,0.9308262099886404) (0.8457350125313283,0.9310074720221851) (0.8455705050239034,0.9311708935286416) (0.8454942355889724,0.931246321711702) (0.8452534586466165,0.9314855244401399) (0.845221951530121,0.9315169470159663) (0.8450126817042607,0.9317247368655838) (0.8448727575290765,0.9318643546206882) (0.8447719047619048,0.9319645519259383) (0.8445311278195489,0.9322048995479046) (0.8445229405174066,0.932213100365609) (0.844290350877193,0.932445081252058) (0.8441725180060912,0.9325631681262629) (0.844049573934837,0.9326858623218761) (0.8438215075205093,0.932914541630929) (0.8438087969924812,0.9329272341004752) (0.8435680200501253,0.9331684793425405) (0.8434699266004968,0.9332672044606543) (0.8433272431077694,0.9334102267767845) (0.8431177928004006,0.933621140049288) (0.8430864661654135,0.9336525606766) (0.8428456892230576,0.933894906859684) (0.8427651236891354,0.9339763316835253) (0.8426049122807018,0.9341376206949501) (0.8424119368502387,0.9343327625029634) (0.8423641353383459,0.9343809161913463) (0.84212335839599,0.9346243408879673) (0.8420582498819248,0.9346904155001637) (0.841882581453634,0.9348680208160576) (0.8417040803971403,0.9350492735207315) (0.8416418045112782,0.9351122770508544) (0.8414010275689223,0.9353567596964945) (0.8413494460236172,0.9354093192633993) (0.8411602506265664,0.935601405045509) (0.8409943644039264,0.935770535280126) (0.8409194736842105,0.9358466208101632) (0.8406786967418546,0.9360921425823874) (0.8406388531955323,0.9361329039762047) (0.8404379197994988,0.9363377523046647) (0.8402829300708435,0.9364964076103812) (0.8401971428571429,0.9365839260301172) (0.839956365914787,0.9368304697328624) (0.8399266127172669,0.9368610282949836) (0.839715588972431,0.9370770423991661) (0.839569918837259,0.9372267479960645) (0.8394748120300752,0.9373241721517818) (0.8392340350877193,0.9375717221042535) (0.8392128661483771,0.9375935485335523) (0.8389932581453634,0.9378192559036699) (0.8388554723833321,0.9379614115814136) (0.8387524812030075,0.9380673393867571) (0.8385117042606516,0.9383158813163837) (0.8384977552900374,0.9383303186678283) (0.8382709273182958,0.9385643740614522) (0.8381397326316606,0.9387002511753753) (0.8380301503759399,0.9388134086219007) (0.837789373433584,0.9390629295608239) (0.8377814221866732,0.9390711903412274) (0.837548596491228,0.9393123786974668) (0.8374228417489001,0.9394431172573623) (0.8373078195488721,0.9395623613370935) (0.8370670426065163,0.9398128495216868) (0.8370640091275692,0.939816012870779) (0.8368262656641603,0.9400632521435501) (0.8367049421473596,0.9401898579837304) (0.8365854887218045,0.9403141795347781) (0.8363456586484511,0.9405646332539656) (0.8363447117794486,0.9405656195088417) (0.8361039348370927,0.9408169771745498) (0.8359861764865714,0.9409403191949829) (0.8358631578947369,0.9410688456800833) (0.8356265135330443,0.9413168961762962) (0.835622380952381,0.9413212173753139) (0.8353816040100251,0.941573536954224) (0.8352666876748363,0.9416943444237111) (0.8351408270676691,0.9418263426504192) (0.8349067168146045,0.942072644019615) (0.8349000501253133,0.9420796417857336) (0.8346592731829574,0.9423329149898376) (0.834546618870742,0.9424517749032751) (0.8344184962406015,0.9425866536934799) (0.8341864117774251,0.9428317168711544) (0.8341777192982456,0.9428408757551034) (0.8339369423558897,0.9430950950944098) (0.8338261134846578,0.9432124495772312) (0.8336961654135339,0.9433497623926395) (0.833465741958318,0.9435939525333388) (0.833455388471178,0.9436049026651945) (0.8332146115288219,0.943860061355607) (0.8331053151802024,0.9439762051095117) (0.8329738345864661,0.9441156526387575) (0.8327448511480702,0.9443591865343445) (0.8327330576441103,0.9443717062392363) (0.8324922807017544,0.9446277981103206) (0.8323843678756883,0.9447428758953651) (0.8322515037593985,0.9448843086074241) (0.8320238833928736,0.945127252139417) (0.8320107268170426,0.9451412705218157) (0.8317699498746867,0.9453982899239373) (0.8316634157455381,0.945512294073056) (0.8315291729323309,0.9456557147406861) (0.8313029829957302,0.9458979803629574) (0.831288395989975,0.9459135798630443) (0.831047619047619,0.9461715215733524) (0.8309426032216786,0.946284289536336) (0.8308068421052631,0.9464298557322903) (0.8305822945178328,0.9466711999813782) (0.8305660651629072,0.9466886189060224) (0.8303252882205514,0.9469474780327309) (0.8302220749949066,0.9470586899476856) (0.8300845112781955,0.947206716515453) (0.8298619627799185,0.947446737546733) (0.8298437343358396,0.947466372576613) (0.8296029573934837,0.9477261444610057) (0.8295019760162321,0.9478353207523342) (0.8293621804511279,0.9479862822521459) (0.8291421328635982,0.9482244174011262) (0.829121403508772,0.9482468260745016) (0.828880626566416,0.9485075061900562) (0.8287824514981933,0.9486140051930593) (0.8286398496240601,0.9487685383228213) (0.8284229501126609,0.9490040616919049) (0.8283990726817042,0.9490299648644671) (0.8281582957393484,0.9492915487124551) (0.8280636469161492,0.949394564325772) (0.8279175187969925,0.9495534703154549) (0.8277045601343518,0.9497854903876372) (0.8276767418546366,0.9498157746667077) (0.8274359649122807,0.9500782576675919) (0.8273457080095449,0.9501768170358886) (0.8271951879699249,0.950341064012693) (0.8269871088006263,0.950568521294879) (0.826954411027569,0.9506042414449994) (0.826713634085213,0.9508676188248899) (0.8266287807831527,0.9509605800554943) (0.8264728571428571,0.9511313053757936) (0.8262707422493777,0.951352970075733) (0.8262320802005012,0.951395351391338) (0.8259913032581454,0.9516596180627225) (0.8259130115082878,0.9517456679812984) (0.8257505263157895,0.9519241805239269) (0.8255556068856406,0.9521386502662035) (0.8255097493734336,0.9521890909055978) (0.8252689724310777,0.9524542413414934) (0.8251985467239991,0.9525318932933868) (0.8250281954887217,0.9527196757072656) (0.824841849382769,0.9529253732953441) (0.824787418546366,0.9529854465685916) (0.82454664160401,0.9532514746691856) (0.8244855332382343,0.9533190663747695) (0.8243058646616541,0.9535177772721102) (0.8241296166835903,0.9537129485052088) (0.8240650877192982,0.9537844051067199) (0.8238243107769423,0.9540513040574972) (0.8237741181289812,0.9541069955317294) (0.8235835338345865,0.9543184716161085) (0.8234190560015319,0.9545011831715973) (0.8233427568922306,0.9545859533461998) (0.8231019799498747,0.9548537154664436) (0.8230644487453833,0.95489548701497) (0.8228612030075189,0.9551217451313753) (0.8227103148217256,0.9552898825256025) (0.822620426065163,0.9553900781546114) (0.822379649122807,0.9556586947350704) (0.8223566727088311,0.9556843450415631) (0.8221388721804511,0.95592758413306) (0.8220035409020868,0.9560788497759627) (0.8218980952380952,0.9561967663672035) (0.8216573182957394,0.9564662274955854) (0.8216509379140275,0.9564733718176983) (0.8214165413533835,0.9567359747705799) (0.8212988822743675,0.9568678861322075) (0.8211757644110276,0.957006004695073) (0.8209473925300314,0.9572623675622349) (0.8209349874686717,0.9572763031732459) (0.8206942105263157,0.9575469029183651) (0.820596487245186,0.9576567908286123) (0.82045343358396,0.9578177796119322) (0.8202461850012719,0.9580511305310514) (0.820212656641604,0.9580889152062797) (0.8199718796992481,0.958360354042532) (0.8198965043970325,0.9584453611489478) (0.8197311027568922,0.95863207721573) (0.8195474640485453,0.958839457042199) (0.8194903258145363,0.958904050089376) (0.8192495488721804,0.9591763130394018) (0.81919908258925,0.9592333924520325) (0.8190087719298246,0.9594488830608572) (0.8188513786699797,0.9596271415018496) (0.8187679949874687,0.9597216940803686) (0.8185272180451127,0.959994764041193) (0.8185043709589893,0.96002067819808) (0.8182864411027567,0.960268181956062) (0.8181580781419827,0.9604139764310464) (0.818045664160401,0.9605418327092622) (0.8178125189221431,0.9608070099758463) (0.8178048872180451,0.9608157030510904) (0.8175641102756892,0.9610899577224803) (0.8174677120201593,0.9611997524932427) (0.8173233333333333,0.9613644505263022) (0.8171236761742542,0.9615921775305684) (0.8170825563909775,0.9616391556276028) (0.8168417794486216,0.9619141929053636) (0.8167804301402105,0.9619842585226411) (0.8166010025062657,0.9621895308016355) (0.8164379926913993,0.9623759687926943) (0.8163602255639097,0.9624650743635224) (0.8161194486215538,0.9627408684321268) (0.8160963826188049,0.9627672815533165) (0.815878671679198,0.9630170551688156) (0.8157556187310514,0.9631581699074047) (0.8156378947368421,0.9632934416385764) (0.8154157198544286,0.9635486068491301) (0.8153971177944862,0.963570015645909) (0.8151563408521303,0.9638470032032223) (0.8150767048329166,0.9639385652649157) (0.8149155639097744,0.9641242377347352) (0.8147385925282111,0.9643280179344247) (0.8146747869674186,0.9644016550005996) (0.8144340100250627,0.9646793519250818) (0.8144014018197483,0.9647169375315647) (0.8141932330827067,0.964957440317328) (0.8140651516047277,0.9651052966254985) (0.8139524561403508,0.965235707441352) (0.8137298607981385,0.9654930676816733) (0.813711679197995,0.9655141420751225) (0.8134709022556391,0.9657930238246183) (0.8133955483327796,0.9658802230628556) (0.8132301253132832,0.9660721479900263) (0.8130622331592862,0.9662667350301828) (0.8129893483709273,0.9663514369658422) (0.8127485714285714,0.9666309587512821) (0.8127299342461497,0.9666525757442239) (0.8125077944862156,0.9669109475984788) (0.8123986705797425,0.967037717266054) (0.8122670175438597,0.9671910994337669) (0.8120684611643376,0.9674221315583389) (0.8120262406015037,0.9674714042197107) (0.8117854636591478,0.9677520722701308) (0.8117393250221331,0.9678057904864323) (0.8115446867167919,0.9680330957814031) (0.8114112811932708,0.9681886658194853) (0.8113039097744361,0.9683142718640869) (0.811084348735859,0.9685707292315653) (0.8110631328320802,0.9685955912680848) (0.8108223558897244,0.9688773861100917) (0.8107585467259932,0.9689519523027907) (0.8105815789473684,0.9691594432878834) (0.8104338942577749,0.9693323065204709) (0.8103408020050126,0.9694416449229509) (0.8101104104433338,0.969711763280266) (0.8101000250626567,0.9697239826321944) (0.8098592481203007,0.9700068708985568) (0.8097881144128444,0.9700902938873471) (0.8096184711779448,0.9702899652380874) (0.8094670253145496,0.9704678695575817) (0.8093776942355889,0.9705731983494659) (0.8091471623147753,0.9708444614187163) (0.8091369172932331,0.9708565627945626) (0.8088961403508772,0.9711404946594905) (0.808828544597952,0.9712200405115804) (0.8086553634085213,0.971424633868291) (0.808511191366632,0.9715945777912971) (0.8084145864661654,0.9717089086800694) (0.8081951218415068,0.9719680441285057) (0.8081738095238096,0.9719933126748663) (0.8079330325814537,0.9722782188101025) (0.8078803552614269,0.9723404103105959) (0.8076922556390977,0.9725634142627096) (0.8075669108834163,0.9727116470429513) (0.8074514786967418,0.9728487446976999) (0.8072548079826916,0.9730817249502063) (0.8072107017543859,0.9731342047772842) (0.8069699248120301,0.9734199927288028) (0.8069440658526783,0.9734506145775119) (0.8067291478696742,0.9737062585675351) (0.8066347038050261,0.9738182863918123) (0.8064883709273183,0.973992661267762) (0.8063267411696269,0.9741847107831345) (0.8062475939849624,0.9742791966224094) (0.8060201972946283,0.9745498580658833) (0.8060068170426065,0.9745658608224664) (0.8057660401002507,0.9748530982484398) (0.8057150915464506,0.9749136984801534) (0.8055252631578949,0.9751405910916042) (0.8054114433098019,0.9752762021930472) (0.8052844862155388,0.9754282219614736) (0.8051092719876913,0.9756373393000033) (0.8050437092731829,0.9757159882308015) (0.8048085970014458,0.975997079826139) (0.8048029323308271,0.9760038877095509) (0.8045621553884712,0.9762924337670917) (0.8045094377907229,0.9763553937275985) (0.8043213784461153,0.9765811791184027) (0.8042118138135244,0.9767122508929121) (0.8040806015037594,0.9768700684106609) (0.803915744546211,0.9770676211443678) (0.8038398245614035,0.9771591006340996) (0.8036212494835147,0.9774214742393909) (0.8035990476190477,0.9774482752456037) (0.8033582706766917,0.9777379156031154) (0.803328348138553,0.9777737798719335) (0.8031174937343357,0.9780279436689026) (0.8030370600428395,0.9781245076738716) (0.8028767167919799,0.9783181256351222) (0.8027474047462987,0.9784736272164185) (0.802635939849624,0.97860846204918) (0.8024594018172765,0.9788211080115397) (0.8023951629072682,0.9788989539351194) (0.8021730708425526,0.9791669195133813) (0.8021543859649123,0.9791896028011633) (0.8019136090225564,0.9794807178609332) (0.8018884314273527,0.9795110311197086) (0.8016728320802005,0.979772221829456) (0.8016055031953585,0.9798534121733504) (0.8014320551378447,0.9800638941377217) (0.8013243057887195,0.9801940319636551) (0.8011912781954887,0.9803557371177947) (0.8010448588680636,0.9805328597279551) (0.8009505012531328,0.9806477535529873) (0.8007671821125081,0.9808698646530412) (0.8007097243107769,0.9809399466762836) (0.8004912952196688,0.9812050158766419) (0.800468947368421,0.9812323201669481) (0.8002281704260652,0.9815250317939281) (0.8002172179056704,0.9815382824889152) (0.7999873934837093,0.981818234447153) (0.7999449699051564,0.9818696335339483) (0.7997466165413534,0.9821116253685642) (0.7996745709712992,0.9821990380112651) (0.7995058395989975,0.9824052083301821) (0.7994060408758069,0.9825264648773397) (0.7992650626566415,0.9826989874161507) (0.7991393994089349,0.982851883047124) (0.7990242857142857,0.9829929669975823) (0.7988746663794924,0.9831752613955758) (0.7987835087719298,0.9832871517027091) (0.7986118616148523,0.9834965687592019) (0.7985427318295739,0.9835815463817239) (0.7983510049609588,0.9838157739376044) (0.798301954887218,0.9838761560656333) (0.7980921162823346,0.9841328456950359) (0.7980611779448622,0.984170985918333) (0.7978352154620894,0.9844477527619628) (0.7978204010025063,0.9844660411810435) (0.7975803224019272,0.9847604638366365) (0.7975796240601504,0.9847613271081057) (0.7973388471177945,0.9850570532736383) (0.7973274570221531,0.9850709475866695) (0.7970980701754384,0.9853530048079504) (0.7970766392616807,0.9853791726506218) (0.7968572932330827,0.985649171635284) (0.7968278890780383,0.9856851076395923) (0.7966165162907268,0.985945555454707) (0.7965812264473756,0.9859887211388172) (0.7963757393483709,0.9862421573541812) (0.79633667136447,0.9862899817092762) (0.796134962406015,0.9865389776601488) (0.7960942438427323,0.986588857889305) (0.7958941854636592,0.9868360157668064) (0.7958539639142135,0.9868853181962134) (0.7956534085213033,0.9871332699424274) (0.7956158516296091,0.9871793311279111) (0.7954126315789474,0.9874307371097371) (0.7953799270582652,0.9874708651645392) (0.7951718546365915,0.987728412596926) (0.7951462102881834,0.9877598887701069) (0.7949310776942355,0.9880262898554152) (0.7949147214260263,0.9880463703941359) (0.7946903007518797,0.9883243601399321) (0.7946854805971209,0.9883302784733083) (0.7944585079454648,0.988611581433124) (0.7944495238095238,0.9886228218704536) (0.7942338236337304,0.9888902476895592) (0.7942087468671679,0.9889216296803512) (0.7940114478432673,0.9891662456507324) (0.793967969924812,0.9892206601267006) (0.7937914007741081,0.9894395437185761) (0.7937271929824561,0.9895198960837105) (0.7935737026449716,0.9897101102905131) (0.7934864160401003,0.9898193160138016) (0.7933583736932666,0.9899779137611375) (0.7932456390977444,0.9901188932784419) (0.7931454341750942,0.9902429225238993) (0.7930048621553885,0.9904185953525294) (0.7929349043652534,0.9905051049727981) (0.7927640852130325,0.9907183829282238) (0.7927268045572413,0.990764429504074) (0.7925233082706766,0.991018208891859) (0.7925211550632582,0.9910208645179104) (0.7923179762142096,0.9912743784201365) (0.7922825313283208,0.991319051682855) (0.7921172883597083,0.9915249396239341) (0.7920417543859649,0.9916199919209066) (0.7919191118680776,0.9917725165515495) (0.791800977443609,0.9919208957791095) (0.791723467126353,0.9920170776360089) (0.7915602005012531,0.9922216773786371) (0.7915303745402847,0.9922585913228377) (0.7913398545343394,0.9924970260717811) (0.7913194235588973,0.9925229059595061) (0.7911519275517018,0.9927323503585308) (0.7910786466165414,0.9928249043570394) (0.7909666140542764,0.9929645326764519) (0.7908378696741856,0.9931265789414638) (0.7907839345226894,0.9931935415383162) (0.7906039094562898,0.9934193454780352) (0.7905970927318295,0.993428018017183) (0.7904265593731505,0.9936419130523962) (0.7903563157894736,0.99373088778332) (0.7902519048100698,0.9938612128428026) (0.7901155388471178,0.9940330603433678) (0.7900799663225722,0.9940772134570149) (0.7899107644849088,0.9942898835308938) (0.7898747619047619,0.9943357209288828) (0.7897443198900591,0.9944991917301451) (0.789633984962406,0.9946387464826585) (0.789580653149731,0.9947051067520691) (0.7894197848943609,0.9949075973273068) (0.7893932080200501,0.9949415447745922) (0.7892617357731146,0.9951066322215906) (0.7891524310776943,0.9952450658743026) (0.789106526453889,0.9953021802374982) (0.7889541776233093,0.9954942102162014) (0.7889116541353384,0.9955485497468657) (0.7888047099867316,0.9956826910392228) (0.7886708771929825,0.9958518108833795) (0.7886581442682419,0.9958675916301877) (0.7885145012106565,0.9960488809565811) (0.7884301002506265,0.996156315062436) (0.788373801575521,0.9962265280315004) (0.788236066143112,0.9964005019154143) (0.7881893233082706,0.9964604007248418) (0.7881013157124348,0.9965707717179153) (0.7879695711012231,0.9967373065994759) (0.7879485463659148,0.9967644053741611) (0.7878408531459403,0.9969000757732053) (0.7877151827017773,0.9970590485066031) (0.7877077694235589,0.9970686460475424) (0.7875925806426529,0.9972141941233136) (0.7874730678612127,0.9973654820048785) (0.787466992481203,0.9973733668669627) (0.7873566652688284,0.9975128815924907) (0.7872433937955976,0.9976563623887456) (0.7872262155388471,0.9976786474753455) (0.7871332743903424,0.9977958939593895) (0.7870263280206086,0.9979314459350695) (0.7869854385964912,0.9979842437827199) (0.7869225756726658,0.9980629880130809) (0.7868220383515054,0.9981904899591107) (0.7867446616541354,0.998289309367645) (0.7867247370808403,0.9983139216089824) (0.7866306929031032,0.9984332528703947) (0.7865399268794466,0.998548453724663) (0.7865038847117795,0.9985952179338043) (0.7864524600897417,0.9986594942284531) (0.7863683136325766,0.9987663445155159) (0.7862875086252561,0.9988689747984172) (0.7862631077694235,0.9989008905855222) (0.7862100662038001,0.9989673553702658) (0.7861360075229433,0.9990614566064377) (0.7860653537561335,0.999151248966297) (0.7860223308270676,0.9992069375722021) (0.7859981260955308,0.999236702994912) (0.7859343457520065,0.9993177893247712) (0.7858740339551429,0.9993944786774918) (0.7858172119532313,0.9994667418655252) (0.7857815538847118,0.9995130858322843) (0.7857639010132708,0.9995345497938578) (0.7857141224209688,0.9995978734617103) (0.7856678974807397,0.9996566839642286) (0.7856252475157027,0.9997109524941709) (0.7855861938676825,0.9997606503435919) (0.7855507578972077,0.9998057489055194) (0.7855407769423559,0.999819451531433) (0.7855189609835107,0.9998462196756277) (0.7854908245245262,0.9998820342539034) (0.7854663699368918,0.9999131643463083) (0.7854456186559455,0.9999395817664333) (0.7854285921357279,0.9999612584371499) (0.7854153118489787,0.9999781663922518) (0.7854057992871384,0.9999902777780938) (0.7854000759603482,0.9999975648552224) (0.7853981633974483,1.0)
			};

			\draw[thick,domain = 1.762/2:2.634/2]    plot (\x,{(4/3.14159)*(exp(2*\x)-1)/(exp(2*\x)+1)}  ) ;
			\draw[thick,domain = 2.634/2:1.8]    plot (\x,{1.10266+(\x-2.634/2)/(3.14159)}  ) ;
			
		\end{tikzpicture}
		\caption{A sketch of the graph of the function $\beta \mapsto \alpha_{sys}(K,C_\beta)$. It is strictly decreasing for $\beta\leq \beta_0$, attains its minimum of $\frac{2\sqrt{2}}{\pi}$ at $\beta=\beta_0$ and is strictly increasing for $\beta\geq \beta_0.$}
		\label{Fig: Graph of C mapsto alpha_sys}
	\end{figure}
	
	For a conformal class $C\in \mathpcal C(K)$, we denote by $g^C_{flat}$ the unique flat metric in $C$ normalized in such a way that the shortest closed horizontal geodesic is of length $\pi$ (see \Cref{Lemma: Characterization of CK} and the discussion before \Cref{Rem: motivation for construction}). Then the minimizer $g_C$ of the systolic area within $C$ can be written as 
	\begin{equation}
		g_C=\phi_C^2g^C_{flat},
	\end{equation}
	for a suitable positive, continuous (but not necessarily smooth) function $\phi_C$ on $K$. We denote by $P_C:L^2(K)\to L^2(K)$ the $L^2$-orthogonal projection onto the line $\R\phi_C\subset L^2(K)$ spanned by $\phi_C$. Here, the $L^2$-scalar product of functions on $K$ is taken with respect to the Riemannian measure induced by $g^C_{flat}$. Now the main results of \Cref{Sec: Klein bottle inequality} can be stated as:
	\begin{ThmIntro}\label{Thm: Main Theorem Klein bottle}
		Let $C\in \mathpcal C(K)$ be a conformal class and $g=\phi^2g^C_{flat}\in C$ be a Riemannian metric on $K$. Then 
		\begin{equation}
			\frac{\| \phi- P_C(\phi)\|^2_{L^2}}{\sys^2(K,g)}\leq \alpha_{sys}(K,g)-\alpha_{sys}(K,C).
		\end{equation}
		In particular, $\alpha_{sys}(K,g)=\alpha_{sys}(K,C)$ holds if and only if $g$ is a constant multiple of $g_C$.
	\end{ThmIntro}
	\begin{ThmIntro}\label{Thm: main Corollary Klein bottle}
		Let $C\in \mathpcal C(K)$ be a conformal class and $g=\phi^2g^C_{flat}\in C$ be a Riemannian metric on $K$. Then 
		\begin{equation}
			\alpha_{sys}(K,C)-\alpha_{sys}(K)+\frac{\| \phi- P_C(\phi)\|^2_{L^2}}{\sys^2(K,g)}\leq \alpha_{sys}(K,g)-\frac{2\sqrt{2}}{\pi}.
		\end{equation}
		In particular, $\alpha_{sys}(K,g)=\frac{2\sqrt{2}}{\pi}$ holds if and only if $C=C_0$ and if $g$ is a constant multiple of $g_{C_0}$.
	\end{ThmIntro}
	\Cref{Thm: Main Theorem Klein bottle} refines the conformal systolic inequality present in \cite{Bavard1988}. It implies \eqref{Eq: Bavards conformal equation} and gives a statement about the stability of the conformal systolic inequality, because it shows that any metric $g$ for which $\alpha_{sys}(K,g)$ is close to the conformal minimum $\alpha_{sys}(K,C)$ is necessarily close to a minimizing metric in the sense that their conformal factors are $L^2$-close. 
	
	Similarly, \Cref{Thm: main Corollary Klein bottle} refines the original systolic inequality \eqref{Eq: Bavards original ineq} for the Klein bottle, established first in \cite{Bavard1986}. Here, the stability can be seen in the following way: if a metric $g$ has systolic area close to $\alpha_{sys}(K)=\frac{2\sqrt{2}}{\pi}$, then the conformal class $C$ containing $g$ has to be close to $C_0$ in the sense that $\alpha_{sys}(K,C)$ is close to $\frac{2\sqrt{2}}{\pi}$, and $g$ has to be close to $g_C$ in the above sense. 
	The main ingredient for the proof of \Cref{Thm: Main Theorem Klein bottle,Thm: main Corollary Klein bottle} is the projection inequality (\Cref{Lemma: Projection Lemma}, proved in \Cref{Sec: Projection Lemma}), which implies that the projection $P_C$ does not decrease the systole. 
	
	In \Cref{Sec: Mobius inequality}, we perform almost the same considerations for the Möbius strip $M$. 
	Since every Klein bottle contains embedded Möbius strips, it is not surprising that the methods for the Klein bottle and the Möbius strip have many similarities. 
	In fact, some of the results of \Cref{Sec: Mobius inequality} are needed in the proof of the theorems about the Klein bottle in \Cref{Sec: Klein bottle inequality}. 
	For $M$, the space $\mathpcal C(M)$ of conformal classes is one-dimensional and can be parametrized by $(0,+\infty)$. Every conformal class $C\in \mathpcal C(M)$ contains a minimizer $g_C$ of the systolic area within $C$, i.e. a metric such that 
	\begin{equation}
		\alpha_{sys}(M,g_C)=\alpha_{sys}(M,C),
	\end{equation} which is unique up to rescaling. We again distinguish suitably normalized metrics $g_{flat}^C\in C$ of constant curvature in every conformal class, write 
	\begin{equation}
		g_C=\phi_C^2g_{flat}^C
	\end{equation}
	for a suitable continuous, not necessarily smooth function $\phi_C$ on $M$, denote by $P_C$ the $L^2(M)=L^2(M,g^C_{flat})$-orthogonal projection onto the line spanned by $\phi_C$ and thus can state the main result of \Cref{Sec: Mobius inequality}: 
	\begin{ThmIntro}\label{Thm: Main result, Mobius strip}
		Let $C\in \mathpcal C(M)$ be a conformal class and $g=\phi^2 g_{flat}^C$ be a Riemannian metric on $M$. Then 
		\begin{equation}
			\frac{\| \phi- P_C(\phi)\|^2_{L^2}}{\sys^2(M,g)}\leq \alpha_{sys}(M,g)-\alpha_{sys}(M,C).
		\end{equation}
		In particular, $\alpha_{sys}(M,g)=\alpha_{sys}(M,C)$ holds if and only if $g$ is a constant multiple of $g_C$.
	\end{ThmIntro}
	This refines the conformal systolic inequality for the Möbius strip due to Pu, present in \cite{Pu1952} by including a stability estimate, similar to the result on the Klein bottle.
	Since $\inf_{C}\alpha_{sys}(M,C)=0$, where the infimum runs over all conformal classes on $M$, the Möbius strip does not have a systolic inequality when the conformal class is not fixed, hence a theorem analogous to \Cref{Thm: main Corollary Klein bottle} does not exist for the Möbius strip.

	\numberwithin{Def}{section}
	\numberwithin{equation}{section}
	\counterwithin{figure}{section}
	\section{The Möbius strip}\label{Sec: Mobius inequality}
	For $\beta>0$, consider $S_\beta=\R \times [-\beta, \beta]\subset \R^2$ and let $g_{flat}$ be the Riemannian metric on $S_\beta$ obtained by restricting the standard Euclidean metric to $S_\beta$. The diffeomorphism $A: \R^2 \to \R^2,$ $A(x,y)=(x+\pi,-y)$ generates the subgroup $G=\{A^k|k\in \Z\}\subset \Diff(\R^2)$. Since $A$ restricts to an isometry of $(S_\beta, g_{flat})$, $G$ acts isometrically on $S_\beta$.  Moreover, $G$ is a discrete group acting freely and properly on $S_\beta$, so $M_\beta:=S_\beta/G$ is a smooth, two-dimensional manifold, which we call the \textit{Möbius strip of width $2\beta$}. The fact that $G$ acts via isometries implies the existence of a unique Riemannian metric on $M_\beta$ such that the quotient map $p: S_\beta \to M_\beta$ is a local isometry. We denote this metric also with $g_{flat}$. Thus, $(M_\beta, g_{flat})$ is a Riemannian manifold, which we call the \textit{flat Möbius strip of width $2\beta$}. 
	Its isometry group is compact and consists of translations in the $x-$direction (induced by translations on $S_\beta$). Note that translation by $\pm\pi$ has the same effect as reflecting along the $x$-axis. 
	By the conformal representation theorem, any Riemannian metric on an arbitrary  Möbius band is conformally equivalent to $(M_\beta, g_{flat})$ for some $\beta>0$, thus it is sufficient to only consider Möbius strips of the form $(M_\beta, \phi^2 g_{flat})$, where $\phi: M_\beta \to (0,+\infty)$ is a smooth positive function. Consequently, the space $\mathpcal C(M) $ of conformal classes of metrics on the Möbius strip is one-dimensional and can be parametrized by $(0,+\infty)$ via the map $\beta \mapsto C_\beta$, where $C_\beta$ is defined as the conformal class of $(M_{\beta},g_{flat})$. To emphasize the membership to the conformal class, we sometimes also write $g_{flat}=g_{flat}^{C_\beta}$.
	
	unique function $\phi$ defined on $S_\beta$ such that $\hat\phi=\psi$. For these reasons, we will identify functions on $M_\beta$ with $G$-invariant functions on $S_\beta$ whenever suitable.  
	
	It is well-known that the projection onto metrics which are invariant under the isometry group of $(M_\beta,g_{flat})$ does not increase the systolic area. 
	More generally, let $(N,g)$ be a (nonsimply connected) two-dimensional Riemannian manifold, $\phi:N \to (0,+\infty)$ a positive smooth function such that $(N,\phi^2 g)$ has finite area and $H$ a compact subgroup of the isometry group $\Isom(N,g)$. As $H$ is a compact group, it carries a unique invariant probability measure $\eta$, called the Haar measure. The Riemannian metric $g$ induces a canonical measure on $N$, and the assumption $\area(N,\phi^2g)<+\infty$ is equivalent to $\phi \in L^2(N,g)$. 
	Using $H$, we construct a map $P_H:L^2(N,g)\to L^2(N,g)$, $\phi \mapsto P_H(\phi)$ defined by 
	\begin{equation}
		P_H(\phi)(p)=\int_H \phi(\xi(p))d\eta(\xi),
	\end{equation}
	for $p\in N$. A direct calculation involving the Fubini-Tonelli theorem and the invariance of the measures shows that $P$ is the $L^2$-orthogonal projection onto the subspace $L^2_H(N,g)=\{\phi \in L^2(N,g)| \phi\circ \xi = \phi\ \forall \xi \in H\}$ of $H$-invariant functions. Therefore
	\begin{equation}
		(P_H(\phi),\phi - P_H(\phi))_{L^2}=0,
	\end{equation}
	where $(\cdot,\cdot)_{L^2}$ denotes the scalar product on $L^2(N,g)$.  Also, if $\phi^2 g$ is a Riemannian metric (i.e. $\phi>0$ and smooth), then $P_H(\phi)$ is smooth and positive, thus $P_H(\phi)^2 g$ is again a Riemannian metric on $N$.
	When acting non-trivially, this projection reduces the systolic area:
	
	\begin{Lemma}\label{Lemma: isometry projection}
		Let $\phi : N\to (0,+\infty)$ be a smooth positive function such that $\area(N,\phi^2g)$ is finite. Then the Riemannian metric $P_H(\phi)^2 g$ satisfies
		\begin{align}
			\area(N,\phi^2 g ) &= \area (N,P_H(\phi)^2 g)+ \|\phi-P_H(\phi)\|^2_{L^2}		\label{Eq: Equality in isometry-projection}	\\
			\sys(N,\phi^2 g) &\leq \sys (N,P_H(\phi)^2 g).									\label{Eq: Projection-ineq between systoles}
		\end{align}
	\end{Lemma}
	\begin{proof}
		The equality in \eqref{Eq: Equality in isometry-projection} follows from the facts that $P_H$ is an $L^2$-orthogonal projection and that the area of a conformal metric is given by the squared $L^2$-norm of the conformal factor:
		\begin{align*}
			\area(N,\phi^2 g)&= \|\phi\|^2_{L^2}=\|P_H(\phi)+\phi-P_H(\phi)\|^2_{L^2}\\
			&=\|P_H(\phi)\|^2_{L^2}+\|\phi-P_H(\phi)\|^2_{L^2}\\
			&=\area (N,P_H(\phi)^2 g)+ \|\phi-P_H(\phi)\|^2_{L^2}.
		\end{align*}
		The inequality between the systoles also follows from a calculation: let $\gamma:[0,1]\to N$ be a smooth noncontractible closed curve. Using the fact that every $\xi \in H$ is an isometry of $(N,g)$, we find
		\begin{align*}
			L_{P_H(\phi)^2 g}(\gamma) 	&=\int_0^1 \sqrt[]{P_H(\phi)^2 ( \gamma(t))\,  g_{\gamma(t)}(\gamma'(t),\gamma'(t))} dt \\
			&=\int_0^1 \int_H \phi(\xi(\gamma(t)))d\eta(\xi)\ 	\sqrt[]{g_{\gamma(t)}(\gamma'(t),\gamma'(t))} dt\\
			&=\int_H \int_0^1  \phi(\xi(\gamma(t)))\	\sqrt[]{g_{\gamma(t)}(\gamma'(t),\gamma'(t))} dt d\eta(\xi) \\
			&=\int_H \int_0^1  \phi(\xi\circ\gamma(t))\	\sqrt[]{g_{\xi\circ\gamma(t)}((\xi\circ\gamma)'(t),(\xi\circ\gamma)'(t))} dt d\eta(\xi) \\
			&=\int_H L_{\phi^2 g}(\xi\circ \gamma) d\eta(\xi) \\
			&\geq \int_H \sys(N,\phi^2 g) d\eta(\xi) = \sys(N,\phi^2 g).
		\end{align*}
		Here, the inequality $L_{\phi^2 g}(\xi\circ \gamma)\geq \sys(N,\phi^2 g)$ holds, because $\xi \circ \gamma$ is again a smooth noncontractible closed curve. Now, by taking the infimum over all noncontractible closed curves $\gamma$, we obtain \eqref{Eq: Projection-ineq between systoles}.
	\end{proof}
	The result of \Cref{Lemma: isometry projection}, (albeit with an inequality and without the term $\|\phi-P_H(\phi)\|^2_{L^2} $ in \eqref{Eq: Equality in isometry-projection}) is the first step in the proofs of Loewner's Torus inequality, the systolic inequality for $\RP^2$  (see \cite{Pu1952}) and Bavard's inequality for the Klein bottle (see \cite{Bavard1986}). In its present form, and with $H=\Isom(N,g)$, it also can be used to prove the estimates on the systolic defect on the torus and $\RP^2$ which appear in \cite{Horowitz2009} and \cite{Katz2020}. Because on the flat Möbius strip and Klein bottle, the isometry group does not act transitively, it is necessary to employ a second projection in order to minimize the systolic area, which will be discussed below and in \Cref{Sec: Klein bottle inequality}.
	
	\begin{figure}
		\centering
		\begin{tikzpicture}[domain=-1:6,scale=1.5]

			\draw[->] (-1,0) -- (6.3,0) node[right] {$y$};
			\draw[->] (0,-.3) -- (0,2.6);
			
			\draw[thin](2,-.08) -- (2,.08);	
			\draw[thin](4,-.08) -- (4,.08);
			\draw[thin](6,-.08) -- (6,.08);			
			
			\node[anchor=north] at  (2,-.1){$1$};	
			\node[anchor=north] at  (4,-.1){$2$};	
			\node[anchor=north] at  (6,-.1){$3$};	
			
			\draw[dashed, thin] (-1,1)--(6.3,1);
			\node[anchor=east] at  (-1,1){$\frac{1}{2}$};	
			
			\draw[dashed, thin] (-1,2)--(6.3,2);
			\node[anchor= east] at  (-1,2){$1$};	
			
			\draw[dashed, thin] (1.762,-.1) -- (1.762,1.414);
			\node[anchor=north] at  (1.762,-.1){$\beta_0$};	
			
			\draw[dashed, thin] (2.634,-.1) -- (2.634,1);	
			\node[anchor=north] at  (2.634,-.1){$\beta_1$};
			
			\node[anchor=south west] at (2,1.3) {$\phi_0$};
			\draw[thick]    plot (\x,{4*exp(\x/2)/(1+exp(\x))}  ) ;
			
		\end{tikzpicture}
		\caption{The graph of the function $y\mapsto \phi_0(y)$. $\beta_1:=\log(2+\sqrt[]{3})$ satisfies $\phi_0(\beta_1)=\frac{1}{2}$.
			$\beta_0:=\log(1+\sqrt[]{2})$ satisfies $\int_0^{\beta_0}\phi_0(y)dy=\frac{\pi}{4}$ and becomes important on the Klein bottle (see \Cref{Sec: Klein bottle inequality}).}
		\label{Fig: Graph of phi_0}
	\end{figure}

	As mentioned above, the isometry group $H_\beta:=\Isom(M_\beta,g_{flat})$ consists of (the maps induced by) translations in the $x$-direction, and includes the reflection along the $x$-axis. Thus, a metric $g=\phi^2g_{flat}^{C_\beta}$ is $H_\beta$-invariant if and only if $\phi$ is constant in the $x$-coordinate and the function $y\mapsto \phi(x,y)$ is even. 
	On $M_\beta$, Pu constructed a special, $H_\beta$-invariant metric $g_0:=\phi_0^2 g_{flat}$ using the function 
	\begin{equation}\label{Eq: Definition of phi_0}
		\phi_0:M_\beta \to (0,+\infty),\quad \phi_0(x,y)=\frac{2e^y}{1+e^{2y}}.
	\end{equation}
	Since $\phi_0$ is constant in $x$, we also use the symbol $\phi_0$ for the function $[-\beta,\beta]\to (0,+\infty)$, $y\mapsto \phi_0(y)=\phi_0(x,y)$. 
	It is smooth, even, decreasing for $y\geq 0 $ and goes to zero for $y\to \pm \infty$. Its graph is sketched in \Cref{Fig: Graph of phi_0}.
	
	Pu also defined a family of curves $\{\gamma_\tau\}_{\tau \in [0,\beta)}$ given by
	\begin{align}
		\gamma_0:[0,\pi]\to M_\beta, \quad \gamma_0(t)=p(t,0)
	\end{align}
	and
	\begin{align}
		\gamma_\tau : [-\tau,\tau] \to M_\beta, \quad \gamma_\tau (t) = p  \left(\pi/2 + \int_0^t \frac{\phi_0(\tau)}{\sqrt{\phi_0^2(s)-\phi_0^2(\tau)}}ds ,-t\right)
	\end{align}
	for $\tau>0$. Here, $p:S_\beta\to M_\beta$ denotes the quotient map. Each $\gamma_\tau$ is a closed, noncontractible curve of length $L_{g_0}(\gamma_\tau)=\pi$. The metric $g_0$ and the curves $\gamma_\tau$ have a clear geometric interpretation:  using the map 
	\begin{align*}
		F:[-\pi/2,\pi/2]\times [-\beta,\beta]&\to [-\pi/2,\pi/2]\times [-\arcsin(\tanh\beta),\arcsin(\tanh \beta)], \\
		(x,y)&\mapsto (x,\arcsin(\tanh y))
	\end{align*}
	and spherical coordinates 
	\begin{align*}
		\Psi:\R^2 &\to S^2\\
		(\theta_1,\theta_2)&\mapsto ( \cos(\theta_1)\cos(\theta_2),\sin(\theta_1)\cos(\theta_2),\sin(\theta_2) )
	\end{align*}
	on $S^2$, one can show that $(M_\beta, g_0)$ is isometric to the Möbius strip obtained from the region $\Psi( [-\pi/2,\pi/2 ]\times [-\Theta_\beta,\Theta_\beta])\subset S^2$ by identifying antipodal points on the boundary $\Psi( \{\pm \pi/2 \}\times [-\Theta_\beta,\Theta_\beta])$, together with the (restriction of the) standard round metric on $S^2$. Here, $\Theta_\beta:=\arcsin(\tanh\beta)\in (0,\pi/2)$. Under this isometry, the curves $\gamma_\tau$ are mapped to segments of great circles connecting the antipodal points on the boundary (compare \Cref{Fig: Isometry of Mobius strip}). 
	
	\begin{figure}
		\centering
		
		\begin{tikzpicture}

			\begin{scope}[,decoration={
					markings,
					mark=at position 0.45 with {\arrow{<}}}
				] 
				\draw[very thick, postaction={decorate}] (6,-.2)--(6,3.2);
				\draw[very thick, postaction={decorate}] (0,3.2)--(0,-.2);
				
			\end{scope}
			\begin{scope}[decoration={
					markings,
					mark=at position 0.1 with {\arrow{>}}}
				] 
				\draw[very thick,postaction={decorate}] (14,1.5) arc(0:-41:3cm);
				\draw[very thick,postaction={decorate}] (8,1.5) arc(180:180-41:3cm);
				
			\end{scope}
			\draw[very thick](0,-.2)--(6,-.2);
			\draw[very thick] (0,3.2)--(6,3.2);	
			\draw[thin](0,1.5)--(6,1.5);

			\draw[very thick] (14,1.5) arc(0:41:3cm);
			
			\draw[thin] (14,1.5) arc(0:360:3cm);

			\draw[very thick] (8,1.5) arc(180:180+41:3cm);
			
			\draw (14,1.5) arc(0:-180:3cm and .9cm);
			\draw [dotted](14,1.5) arc(0:180:3cm and .9cm);

			\draw[dotted] (11+3*0.7549,1.5+3*0.6557 -0.03 ) arc(0:180:3*0.7549cm and .5*0.7549cm);
			\draw[very thick] (11+3*0.7549,1.5+3*0.6557 -0.03 ) arc(0:-180:3*0.7549cm and .5*0.7549cm);
			
			\draw[very thick] (11+3*0.7549,1.5-3*0.6557+0.03) arc(0:-180:3*0.7549cm and .5*0.7549cm);
			\draw[dotted] (11+3*0.7549,1.5-3*0.6557+0.03) arc(0:180:3*0.7549cm and .5*0.7549cm);
			
			\node[anchor=mid] at (7,1.5){$\stackrel{\Psi\circ F}{\to}$};
			\node[anchor=mid] at (3.6,.6){$\gamma_\tau$};
			
			\node at  (8.25,2.7)[circle,fill,inner sep=1.5pt]{};
			\node at  (13.75,0.3)[circle,fill,inner sep=1.5pt]{};
			\draw[dashed,thick,rotate around ={156:(8.25,2.7)}] (8.25,2.7) arc(0:180:3cm and 1.17cm);
			\node[anchor=mid] at (10,1.5){$\Psi \circ F(\gamma_\tau)$};
			
			\draw[dashed,thick] (0,2.5) to[out=0, in=140](3,1.5) to[out=-40, in=180](6,.5);
			\node at  (0,2.5)[circle,fill,inner sep=1.5pt]{};
			\node at  (6,0.5)[circle,fill,inner sep=1.5pt]{};

		\end{tikzpicture}

		\caption{$\Psi\circ F$ maps the fundamental region of the Möbius strip $(M_\beta, g_0)$ isometrically to a subset of the round sphere. The geodesic $\gamma_\tau$ is mapped to a half great circle connecting antipodal points.}
		\label{Fig: Isometry of Mobius strip}
	\end{figure}

	For $\beta\leq \beta_1:=\log(2+\sqrt{3})$, the curves $\gamma_\tau$ realize the systole on $(M_\beta,g_0)$ and we thus have $\sys(M_\beta,g_0)=\pi$. The fact that no shorter noncontractible closed curve exists can be seen in the following way: closed curves in the free homotopy class of $\gamma_\tau$ (or in the homotopy class of the curve obtained by reversing the orientation of $\gamma_\tau$) have to connect antipodal points in $S^2$ (when considering their image under the isometry $F$), and thus are of length at least $\pi$. For a closed noncontractible curve $\gamma:[0,1]\to M_\beta$ in a different homotopy class, we can consider the lift to a curve $\hat{\gamma}:[0,1]\to S_\beta$ of the form $\hat{\gamma}(t)=(\hat x (t), \hat y(t))$. It satisfies $\hat x (1)- \hat x (0)=\pi k $ with $k\in \Z$, $|k|\geq 2$. The inequality $|k|\geq 2$ holds,  because $\gamma$ is neither homotopic to a point (corresponding to $k=0$), nor to $\gamma_\tau$ or its reverse (corresponding to $k=\pm 1$).
	Furthermore, $\beta\leq \beta_1$ is equivalent to $\phi_0(y)\geq \frac{1}{2}$ for all $y\in [-\beta,\beta]$, and thus 
	\begin{align*}
		L_{g_0}(\gamma)=L_{g_0}(\hat \gamma)=\int_0^1 \underbrace{\phi_0(\hat y (t))}_{\geq 1/2}\ \underbrace{\sqrt[]{\hat x ' (t)^2 + \hat y' (t)^2} }_{\geq |\hat x' (t)|}dt \geq \frac{1}{2}\int_0^1 |\hat x ' (t)| dt \geq \frac{\pi |k|}{2}\geq \pi.
	\end{align*}
	A direct calculation shows that for $\tau \in [-\beta,\beta]$, the \enquote{horizontal} curve $\lambda_\tau:[-\pi/2,3\pi/2]\to M_\beta$, $\lambda_\tau(t)=p(t,\tau)$ is of length $L_{g_0}(\lambda_\tau)=2\pi \phi_0(\tau)$, which is less than $\pi$ for $\tau>\beta_1$. Hence
	\begin{equation}\label{Eq: Systole of Mbeta}
		\sys (M_\beta,g_0)=		\begin{cases}
			\pi, \quad 				&\textnormal{if } \beta \leq \beta_1\\
			2 \pi \phi_0(\beta), 	&\textnormal{if } \beta > \beta_1.
		\end{cases}
	\end{equation}
	The geometric interpretation of this phenomenon is that for $\beta>\beta_1$, i.e. for round Möbius strips whose boundary is \enquote{close enough} to the poles in $S^2$, it is shorter to complete a full rotation near the boundary in the covering region contained in $S^2$ than to travel to the antipodal point in $S^2$. For $\beta>\beta_1$ we call $M_\beta$ a \textit{thick} Möbius strip. 
	
	The curves $\lambda_\tau$ also show that any $H_\beta$-invariant metric $g=\phi^2 g_{flat}$ on $M_\beta$ with systole at least $\pi$ necessarily satisfies $\phi\geq \frac {1}{2}$. This motivates the construction of the metric $g_\beta:=\phi_\beta^2 g_{flat}^{C_\beta}\in C_\beta$ (see also \Cref{Fig: phibeta Mobius with isometry}), where 
	\begin{equation}\label{Eq: Def of conf factor mobius}
		\phi_\beta:M_\beta\to (0,+\infty), \quad \phi_\beta(p(x,y))=\max\left\{\phi_0(y),\frac{1}{2}\right\}=
		\begin{cases} 
			\phi_0(y), \quad 	&\textnormal{if }|y|\leq \beta_1, \\ 
			\frac{1}{2}, 		&\textnormal{if }|y|> \beta_1. 
		\end{cases}
	\end{equation}
	\begin{figure}		
		\mbox{}\hfill  
		\newsavebox{\imageboxone}
		\savebox{\imageboxone}{
			\begin{tikzpicture}[domain=-1:2.634,scale=1]

				\draw[->] (-1,0) -- (5,0) node[right] {$y$};
				\draw[->] (0,-.3) -- (0,2.6);
				
				\draw[thin](2,-.08) -- (2,.08);	
				\draw[thin](4,-.08) -- (4,.08);

				\node[anchor=north] at  (2,-.1){$1$};	
				\node[anchor=north] at  (4,-.1){$2$};

				\draw[dashed, thin] (-1,1)--(2.634,1);
				\node[anchor=east] at  (-1,1){$\frac{1}{2}$};	
				
				\draw[dashed, thin] (-1,2)--(4.9,2);
				\node[anchor= east] at  (-1,2){$1$};

				\draw[dashed, thin] (2.634,-.1) -- (2.634,1);	
				\node[anchor=north] at  (2.634,-.1){$\beta_1$};
				
				\node[anchor=south west] at (2.5,1) {$\phi_\beta$};
				\draw[thick]    plot (\x,{4*exp(\x/2)/(1+exp(\x))}  ) ;
				\draw[thick] (2.634,1)--(4.9,1);
				
			\end{tikzpicture}
		}

		\begin{subfigure}[t]{0.45\linewidth}
			
			\centering\usebox{\imageboxone}
			\caption*{The graph of $\phi_\beta$ for $\beta>\beta_1$.}
		\end{subfigure}
		\hfill
		\begin{subfigure}[t]{0.45\linewidth}
			\centering\raisebox{\dimexpr.5\ht\imageboxone-.5\height}{
				
				\begin{tikzpicture}[scale=.4]
					
					\draw[dotted] (11+3*0.7549,1.5+3*0.6557 -0.03 ) arc(0:180:3*0.7549cm and .5*0.7549cm);
					\draw[] (11+3*0.7549,1.5+3*0.6557 -0.03 ) arc(0:-180:3*0.7549cm and .5*0.7549cm);
					
					\draw[] (11+3*0.7549,1.5-3*0.6557+0.03) arc(0:-180:3*0.7549cm and .5*0.7549cm);
					\draw[dotted] (11+3*0.7549,1.5-3*0.6557+0.03) arc(0:180:3*0.7549cm and .5*0.7549cm);
					
					\begin{scope}[decoration={
							markings,
							mark=at position 0.1 with {\arrow{>}}}
						] 
						\draw[very thick,postaction={decorate}] (14,1.5) arc(0:-41:3cm);
						\draw[very thick,postaction={decorate}] (8,1.5) arc(180:180-41:3cm);
						
					\end{scope}
					\draw[very thick] (14,1.5) arc(0:41:3cm);
					\draw[very thick] (8,1.5) arc(180:180+41:3cm);
					
					\draw (14,1.5) arc(0:-180:3cm and .9cm);
					\draw [dotted](14,1.5) arc(0:180:3cm and .9cm);

					\draw[very thick](11+3*0.7549,1.5+3*0.6557 -0.03 ) --(11+3*0.7549,1.5+3*0.6557 -0.03 + 2) ;
					\draw[very thick](11-3*0.7549,1.5+3*0.6557 -0.03 ) --(11-3*0.7549,1.5+3*0.6557 -0.03 + 2) ;
					\draw[very thick] (11+3*0.7549,1.5+3*0.6557 -0.03 +2) arc(0:180:3*0.7549cm and .5*0.7549cm);
					\draw[very thick] (11+3*0.7549,1.5+3*0.6557 -0.03+2 ) arc(0:-180:3*0.7549cm and .5*0.7549cm);
					
					\draw[very thick](11+3*0.7549,1.5-3*0.6557 -0.03 ) --(11+3*0.7549,1.5-3*0.6557 -0.03 - 2) ;
					\draw[very thick](11-3*0.7549,1.5-3*0.6557 -0.03 ) --(11-3*0.7549,1.5-3*0.6557 -0.03 - 2) ;
					\draw[very thick] (11+3*0.7549,1.5-3*0.6557+0.03-2) arc(0:-180:3*0.7549cm and .5*0.7549cm);
					\draw[dotted] (11+3*0.7549,1.5-3*0.6557+0.03-2) arc(0:180:3*0.7549cm and .5*0.7549cm);
					
				\end{tikzpicture}
			}
			\caption*{For $\beta>\beta_1$, $(M_\beta,g_\beta)$ is isometric to the Möbius strip on a sphere with two cylinders of radius $\frac{1}{2}$ attached.}
		\end{subfigure}
		\hfill

		\caption{By identifying two antipodal meridians on the cylinder-sphere with opposing orientations, one obtains a Möbius strip isometric to $(M_\beta, g_\beta)$. The cylinders attached to the sphere are each of height $\frac{\beta-\beta_1}{2}$. }
		\label{Fig: phibeta Mobius with isometry}
	\end{figure}

	Note that because $\phi_\beta$ is in general not smooth but only continuous, the metric $g_\beta$ is not a smooth metric either. 
	It satisfies 
	\begin{equation}
		\sys(M_\beta,g_\beta)=\pi, \quad \area(M_\beta,g_\beta)=	\begin{cases}
			2\pi\tanh\beta, \quad \textnormal{for } \beta\leq \beta_1\\
			\pi\,\sqrt[]{3}+\frac{\pi}{2}(\beta-\beta_1),\quad \textnormal{for } \beta> \beta_1.
		\end{cases}
	\end{equation}
	It is the aim of the following to show that $g_\beta$ in fact minimizes the systolic area of $M_\beta$ in the conformal class of $g_{flat}$. The key part in the proof is showing that projecting $H_\beta$-invariant metrics onto $g_\beta$ does not decrease the systole: consider the map 
	\begin{equation}
		P_{\phi_\beta}: L^2(M_\beta, g_{flat}) \to L^2(M_\beta, g_{flat}),\quad P_{\phi_\beta}(\phi)=\frac{(\phi,\phi_\beta)_{L^2}}{(\phi_\beta,\phi_\beta)_{L^2}}\phi_\beta.
	\end{equation} 
	It is the $L^2$-orthogonal projection onto the subspace $\R \phi_\beta\subset L^2(M_\beta, g_{flat})$ and projects a positive function $\phi$ to a positive multiple of $\phi_\beta$. Thus, in this case, $P_{\phi_\beta}(\phi)>0$, hence $P_{\phi_\beta}(\phi)^2g_{flat}$ is again a (continuous) Riemannian metric on $M_\beta$. 
	\begin{Lemma}\label{Lemma: Projection Lemma for Mobius}
		Let $g=\phi^2 g_{flat}$ be a $H_\beta$-invariant Riemannian metric on $M_\beta$. Then the metric $P_{\phi_\beta}(\phi)^2 g_{flat}$ satisfies
		\begin{align}
			\area(M_\beta,\phi^2 g_{flat}) &= \area (M_\beta, P_{\phi_\beta}(\phi)^2 g_{flat}) + \|\phi-P_{\phi_\beta}(\phi)\|_{L^2}^2 \label{Eq: Equality in möbius-projection}\\
			\sys(M_\beta,\phi^2 g_{flat})&\leq \sys (M_\beta, P_{\phi_\beta}(\phi)^2 g_{flat})			\label{Eq: projection-ineq for möbius proj}
		\end{align}
	\end{Lemma}
	\begin{proof}
		The proof of \eqref{Eq: Equality in möbius-projection} is identical to the proof of \eqref{Eq: Equality in isometry-projection} from \Cref{Lemma: isometry projection}. It follows from the fact that $P_{\phi_\beta}$ is an $L^2$-orthogonal projection. 
		Due to 
		\begin{equation}
			\sys (M_\beta, P_{\phi_\beta}(\phi)^2 g_{flat})	= \frac{(\phi,\phi_\beta)_{L^2}}{(\phi_\beta,\phi_\beta)_{L^2}} \sys(M_\beta, g_\beta) = \pi \frac{(\phi,\phi_\beta)_{L^2}}{(\phi_\beta,\phi_\beta)_{L^2}},
		\end{equation}
		the inequality \eqref{Eq: projection-ineq for möbius proj} is equivalent to  
		\begin{equation}\label{Eq: proj-ineq reformulated Möbius}
			\frac{(\phi,\phi_\beta)_{L^2}}{(\phi_\beta,\phi_\beta)_{L^2}}\geq \frac{\sys(M_\beta,g)}{\pi}.
		\end{equation}
		As this inequality is invariant under rescalings of $\phi$ by positive constants, it is sufficient to only consider metrics satisfying $\sys(M_\beta,g)=\pi$, for which \eqref{Eq: proj-ineq reformulated Möbius} simplifies to 
		\begin{equation}\label{Eq: proj-ineq simplest}
			(\phi,\phi_\beta)_{L^2}\geq (\phi_\beta,\phi_\beta)_{L^2}.
		\end{equation}
		This inequality essentially follows from an equality of Pu  that is stated and proven on pages 69 and 70 of \cite{Pu1952}. He showed that 
		\begin{equation}\label{Eq: Pus Möbius equation}
			\int_0^\beta \phi(y) \phi_0(y) dy = \frac{1}{\pi} \int_0^\beta L_{g}(\gamma_\tau) \left( -\restr{\frac{d}{ds}}{s=\tau}\, \sqrt[]{\phi_0^2(s)-\phi_0^2(\beta)}\right) d\tau 
		\end{equation} 
		holds for all $\beta>0$ and all $H_\beta$-invariant metrics $g=\phi^2 g_{flat}$ on $M_\beta$, where $\gamma_\tau$ are the curves defined above. $\sys(M_\beta,g)=\pi$ implies $L_g(\gamma_\tau)\geq \pi$ and thus 
		\begin{equation}
			\int_0^\beta \phi(y) \phi_0(y) dy \geq \int_0^\beta \left( -\restr{\frac{d}{ds}}{s=\tau}\, \sqrt[]{\phi_0^2(s)-\phi_0^2(\beta)}\right) d\tau = \tanh \beta.
		\end{equation} 
		Combining this with $\int_0^\beta \phi_0^2(y)dy=\tanh \beta$, we have what we will call Pu's inequality:
		\begin{equation}\label{Eq: Pus inequality}
			\int_0^\beta \phi(y) \phi_0(y) dy  \geq \int_0^\beta \phi_0^2(y) dy. 
		\end{equation} 
		It holds for all $H_\beta$-invariant functions $\phi$ satisfying $\sys(M_\beta,\phi^2g_{flat})\geq \pi$. 
		
		Because $H_\beta$-invariant functions $f,h\in L^2(M,g_{flat})$ are constant in the $x$-coordinate and even in $y$, their scalar product is given by 
		\begin{equation}\label{Eq: L2 product on Möbius}
			(f,h)_{L^2}=\int_{M_\beta} f(x,y)h(x,y) dxdy = 2\pi\int_0^\beta f(y) h(y)dy.
		\end{equation}  
		Hence, it is sufficient to show $\int_0^\beta \phi(y)\phi_\beta(y)\geq \int_0^\beta \phi_\beta^2(y)dy$ in order to prove \eqref{Eq: proj-ineq simplest}. 
		For $\beta\leq \beta_1$, this immediately follows from \eqref{Eq: Pus inequality}, because $\phi_\beta(y)=\phi_0(y)$ for $y\leq \beta_1$. 
		
		For $\beta>\beta_1$, we observe that restricting $g$ to $M_{\beta_1}\subset M_\beta$ cannot decrease the systole (because any noncontractible closed curve in $M_{\beta_1}$ is also a noncontractible closed curve in $M_\beta$), so $\sys (M_{\beta_1},\restr{g}{M_{\beta_1}})\geq \pi$ holds, which implies 
		\begin{equation}\label{Eq: Pus ineq restricted}
			\int_0^{\beta_1}\phi(y)\phi_0(y)dy \geq \int_0^{\beta_1} \phi_0^2(y)dy.  
		\end{equation}
		Furthermore, $\sys(M_\beta,g)\geq \pi$ forces the curves $\lambda_\tau$ to be of length at least $\pi$, hence $\phi\geq \frac{1}{2}$. In particular, we obtain $\phi(y)\geq \phi_\beta(y)=\frac{1}{2}$ for $y\geq \beta_1$. In combination with \eqref{Eq: Pus ineq restricted} this yields
		\begin{align*}
			\int_0^\beta \phi(y)\phi_\beta(y)dy &= \int_0^{\beta_1} \phi(y)\phi_\beta(y)dy + \int_{\beta_1}^\beta \phi(y)\phi_\beta(y)dy \\
			&=\int_0^{\beta_1}\phi(y)\phi_0(y)dy + \int_{\beta_1}^\beta \phi(y)\phi_\beta(y)dy \\
			&\geq \int_0^{\beta_1} \phi_0^2(y)dy + \int_{\beta_1}^\beta \phi_\beta(y)\phi_\beta(y)dy \\
			&= \int_0^{\beta_1} \phi_\beta^2(y)dy + \int_{\beta_1}^\beta \phi_\beta^2 (y)dy \\
			&= \int_0^\beta \phi_\beta^2(y) dy, 
		\end{align*}
		completing the proof. 
	\end{proof}
	The results above can be summarized as follows: 
	\begin{Thm}\label{Thm: Systolic inequality with remainder for the Mobius stip}
		Let $\beta>0$,  $C_{\beta}\in \mathpcal C(M)$ be a conformal class and $g=\phi^2 g_{flat}^{C_\beta}\in C_\beta$ be a Riemannian metric on the Möbius strip $M_\beta$. Then
		\begin{equation}\label{Eq: optimal systolic constant mobius}
			\alpha_{sys}(M,C_\beta)=\alpha_{sys}(M_\beta,g_\beta)=\begin{cases}
				\frac{2}{\pi}\tanh \beta , \quad &\textnormal{if } \beta\leq \beta_1\\
				\dfrac{\sqrt[]{3}}{\pi}+\dfrac{\beta-\beta_1}{2\pi} &\textnormal{if } \beta>\beta_1,
			\end{cases}
		\end{equation}
		and
		\begin{equation}\label{Eq: Systolic inequality with remainder for the Möbius strip}
			\frac{\|\phi- P_{\phi_\beta}(\phi)\|^2_{L^2}}{\sys^2(M_\beta,g)}\leq \alpha_{sys}(M_\beta,g)-\alpha_{sys}(M,C_\beta).			
		\end{equation}
		In particular, $\alpha_{sys}(M_\beta,g)=\alpha_{sys}(M,C_\beta)$ holds if and only if $\phi$ is a positive multiple of $\phi_\beta$. 
	\end{Thm}
	\begin{proof}
		Applying \Cref{Lemma: isometry projection} with $H=H_\beta$ yields 
		\begin{align*}
			\area(M,g ) &= \area (M,P_{H_\beta}(\phi)^2 g)+ \|\phi-P_{H_\beta}(\phi)\|^2_{L^2}	\\
			\sys(M,g) &\leq \sys (M,P_{H_\beta}(\phi)^2 g).
		\end{align*}
		By construction, the metric $P_{H_\beta}(\phi)^2 g_{flat}$ is $H_\beta$-invariant, so we can use \Cref{Lemma: Projection Lemma for Mobius} to find 
		
		\begin{equation*}
			\area(M_{\beta}, g ) = \area (M_{\beta},P_{\phi_\beta}(P_{H_\beta}(\phi))^2 g_{flat})+ \|\phi-P_{H_\beta}(\phi)\|^2_{L^2}+\|P_{H_\beta}(\phi)- P_{\phi_\beta}(P_{H_\beta}(\phi))\|^2_{L^2}
		\end{equation*}
		and 
		\begin{equation*}
			\sys(M_\beta,g)\leq \sys(M_\beta, P_{H_\beta}(\phi)^2 g_{flat})\leq \sys(M_\beta, P_{\phi_\beta}(P_{H_\beta}(\phi))^2 g_{flat}).
		\end{equation*}
		Because $\R\phi_\beta\subset L^2_H(M_\beta)$, we find $P_{\phi_\beta}\circ P_{H_\beta}=P_{\phi_\beta}$, so using the orthogonality of the projections and the Pythagorean Theorem yields
		\begin{equation*}
			\|\phi-P_{H_\beta}(\phi)\|^2_{L^2}+\|P_{H_\beta}(\phi)- P_{\phi_\beta}(P_{H_\beta}(\phi))\|^2_{L^2}=\|\phi-P_{\phi_\beta}(\phi)\|^2_{L^2}.
		\end{equation*}
		Combined, these inequalities give
		\begin{align*}
			\alpha_{sys}(M_\beta,g)= \dfrac{\area(M_\beta,P_{\phi_\beta}(\phi)^2 g_{flat})
				+ \|\phi-P_{\phi_\beta}(\phi)\|^2_{L^2}}{\sys^2(M_\beta,g)}\nonumber \\
			\geq \dfrac{\area (M_\beta,P_{\phi_\beta}(\phi)^2 g_{flat})}{\sys^2(M_\beta,P_{\phi_\beta}(\phi)^2 g_{flat})}+ \dfrac{\|\phi- P_{\phi_\beta}(\phi)\|^2_{L^2}}{\sys^2(M_\beta,g)}, 
		\end{align*}
		or equivalently 
		\begin{equation}\label{Eq: systolic defect mobius strip}
			\dfrac{\|\phi- P_{\phi_\beta}(\phi)\|^2_{L^2}}{\sys^2(M_\beta,g)} \leq \alpha_{sys}(M_\beta,g)-\alpha_{sys}(M_\beta, P_{\phi_\beta}(\phi)^2 g_{flat}).
		\end{equation}
		Because $P_\beta(\phi)$ is a constant multiple of $\phi_\beta$, the scaling-invariance of the systolic area gives us $\alpha_{sys}(M_\beta,P_\beta(\phi)^2g_{flat})=\alpha_{sys}(M_\beta,g_\beta)$, so \eqref{Eq: systolic defect mobius strip} implies $\alpha_{sys}(M_\beta,g)\geq \alpha_{sys}(M_\beta,g_\beta)$ for all $g\in C_\beta$, which proves \eqref{Eq: optimal systolic constant mobius}	and consequently \eqref{Eq: Systolic inequality with remainder for the Möbius strip}. 
		If $\alpha_{sys}(M_\beta,g)-\alpha_{sys}(M,C_\beta)=0$, then $\|\phi- P_{\phi_\beta}(\phi)\|^2_{L^2}=0$ and hence 
		\begin{equation}
			\phi=P_{\phi_\beta}(\phi),
		\end{equation}
		so $\phi$ is a multiple of $\phi_\beta$. Conversely, $\alpha_{sys}(M_\beta,\lambda^2 g_\beta)-\alpha_{sys}(M_\beta,C_\beta)=0$ holds by the scaling-invariance of the systolic area and by \eqref{Eq: optimal systolic constant mobius} for all constants $\lambda>0$.
	\end{proof}
	
	Instead of using $g_{flat}^{C_{\beta}}$ as a reference metric in the conformal class $C_\beta$, we can use the optimal metric $g_{\beta}$ and get a similar result. It induces a measure on $M_{\beta}$, so  we can consider the space $L^2(M_\beta, g_{\beta})$, where the scalar product is taken with respect to this measure. Here, the orthogonal projection $h\mapsto \frac{(h,1)_{L^2}}{(1,1)_{L^2}}$ onto the space of constant functions is precisely the expected value of $h$ with respect to the probability measure $\mu_{g_\beta}$ obtained by normalizing the measure induced by $g_{\beta}$. For a metric $g\in C_\beta$ we can write
	\begin{align*}
		g&=h^2 g_{\beta}\\
		g&=\phi^2 g_{flat}^{C_\beta}
	\end{align*}
	for suitable positive functions $h,\phi$ on $M_\beta$. Because of $g_{\beta}=\phi_{\beta}^2g_{flat}^{C_\beta}$, we find
	\begin{equation}\label{Eq: phi in terms of h}
		\phi=h\phi_{\beta},
	\end{equation} 
	and a direct calculation shows
	\begin{equation}
		(f_1,f_2)_{L^2(M_{\beta},g_{\beta})}=(f_1\phi_\beta, f_2 \phi_\beta)_{L^2(M_{\beta},g_{flat})}
	\end{equation}
	for all functions $f_1,f_2\in L^2(M_\beta, g_{\beta})$.
	Using this and \eqref{Eq: phi in terms of h} we obtain 
	\begin{align*}
		P_{\phi_\beta}(\phi)&=\frac{(\phi,\phi_\beta)_{L^2(M_\beta,g_{flat})}}{(\phi_\beta,\phi_\beta)_{L^2(M_\beta,g_{flat})}}\phi_\beta \\
		&=\frac{(h\phi_\beta,\phi_\beta)_{L^2(M_\beta,g_{flat})}}{(\phi_\beta,\phi_\beta)_{L^2(M_\beta,g_{flat})}}\phi_\beta\\
		&=\frac{(h,1)_{L^2(M_\beta,g_{\beta})}}{(1,1)_{L^2(M_\beta,g_{\beta})}}\phi_\beta \\
		&=E(h)\phi_\beta,
	\end{align*}
	where $E(h)$ denotes the expected value of $h$ with respect to the probability measure $\mu_{g_\beta}$. 
	This leads to 
	\begin{align}
		\| \phi - P_{\phi_\beta}(\phi)\|^2_{L^2(M_{\beta},g_{flat})}&=\| h\phi_\beta - E(h)\phi_\beta \|^2_{L^2(M_{\beta},g_{flat})}
		=\| h - E(h) \|^2_{L^2(M_{\beta},g_{\beta})}\nonumber\\
		&= \area(M_{\beta},g_{\beta}) \Var(h),
	\end{align}
	where $\Var(h)=E(h^2)-E(h)^2$ is the variance with respect to $\mu_{g_\beta}$.
	Thus, as a corollary, we obtain a reformulation of \eqref{Eq: Systolic inequality with remainder for the Möbius strip}, similar to the stability results involving the variance on $\RP^2$ and $\T^2$, as in \cite{Katz2020} and \cite{Horowitz2009}.
	\begin{Cor}\label{Cor: Variance formulation Mobius}
		Let $C_\beta\in \mathpcal C(M)$ be a conformal class and $g=h^2g_{\beta}\in C_{\beta}$ be a Riemannian metric on the Möbius strip $M_{\beta}$. Then 
		\begin{equation}
			\area(M_{\beta},g_{\beta})	\frac{ \Var(h)}{\sys^2(M_{\beta},g)}\leq \alpha_{sys}(M_\beta,g)-\alpha_{sys}(M,C_\beta).
		\end{equation}
		In particular, $\alpha_{sys}(M_{\beta},g)=\alpha_{sys}(M,C_{\beta})$ holds if and only if $h$ is constant. 
	\end{Cor}
	
	\section{The Klein bottle}\label{Sec: Klein bottle inequality}
	The Klein bottle can be constructed as a quotient of the standard Euclidean plane $(\R^2,g_{flat})$ by a discrete group. For $\beta>0$, consider the subgroup $\Gamma_\beta\subset \Isom(\R^2,g_{flat})$ generated by the isometries $A:\R^2\to\R^2$, $A(x,y)=(x+\pi,-y)$ and $B: \R^2\to\R^2$, $B(x,y)=(x,y+4\beta)$. The group $\Gamma_\beta$ acts isometrically, freely and properly on $(\R^2,g_{flat})$, thus $K_\beta:=\R^2/\Gamma_\beta$ is a smooth manifold with a unique Riemannian metric, also denoted by $g_{flat}$, such that the quotient map $p: \R^2\to K_\beta$ is a local isometry. The pair $(K_\beta,g_{flat})$ is called \textit{flat Klein bottle of width $4\beta$}. It contains two Möbius strips of width $2\beta$, as is suggested in \Cref{Fig: Klein Bottle from Mobius}. 
	We denote by $C_\beta$ the conformal class of $(K_\beta,g_{flat})$, i.e. the set of metrics of the form $\phi^2 g_{flat}$, for some positive function $\phi$, and sometimes write $g_{flat}=g_{flat}^{C_\beta}$ when we want to emphasize that $g_{flat}^{C_\beta}$ is an element of $C_\beta$ or when $\beta$ is not clear from the context. 
	The following Lemma shows that up to isometry, the space $\mathpcal C(K)$ of conformal classes of metrics on the Klein bottle can be parametrized by the map $(0,+\infty)\to \mathpcal C(K) $, $\beta\mapsto C_\beta$. It is a direct consequence of the conformal representation theorem and the characterization of the space of flat metrics on the Klein bottle (see for example \cite{Wolf1967}, Proposition 2.5.9).
	\begin{Lemma}\label{Lemma: Characterization of CK}
		Let $g$ be a Riemannian metric on a Klein bottle $K$. Then there exists $\beta>0$ and a smooth function $\phi:K_\beta\to (0,+\infty)$ such that $(K,g)$ is isometric to $(K_\beta,\phi^2 g_{flat})$.
	\end{Lemma} 
	
	\begin{figure}		
		\mbox{}\hfill  
		\newsavebox{\imageboxfour}
		\savebox{\imageboxfour}{
			\begin{tikzpicture}

				\draw[thick,fill=gray!45](0,1)--(6,1)--(6,3)--(0,3)--(0,1);
				\node at (3,2) {$M_\beta$} ;

				\draw[fill=gray!15](0,4)--(6,4)--(6,3)--(0,3)--(0,4);
				\node at (3,3.5) {$M_\beta^-$} ;
				
				\draw[fill=gray!15](0,1)--(6,1)--(6,0)--(0,0)--(0,1);
				\node at (3,.5) {$M_\beta^+$} ;

				\begin{scope}[ thick,decoration={
						markings,
						mark=at position 0.55 with {\arrow{>>>>}}}
					] 
					\draw[postaction={decorate}] (0,0)--(6,0);
					\draw[postaction={decorate}] (0,4)--(6,4);
				\end{scope}

				\begin{scope}[ thick,decoration={
						markings,
						mark=at position 0.5 with {\arrow{>>}}}
					] 
					\draw[postaction={decorate}] (6,1)--(6,3);
					\draw[postaction={decorate}] (0,3)--(0,1);
				\end{scope}

				\begin{scope}[ thick,decoration={
						markings,
						mark=at position 0.6 with {\arrow{>}}}
					] 
					\draw[postaction={decorate}] (6,3)--(6,4);
					\draw[postaction={decorate}] (0,1)--(0,0);
				\end{scope}
				
				\begin{scope}[ thick,decoration={
						markings,
						mark=at position 0.6 with {\arrow{>>>}}}
					] 
					\draw[postaction={decorate}] (6,0)--(6,1);
					\draw[postaction={decorate}] (0,4)--(0,3);
				\end{scope}			
				
			\end{tikzpicture}
		}

		\begin{subfigure}[t]{0.45\linewidth}
			
			\centering\usebox{\imageboxfour}
			\caption*{$K_\beta$ can be decomposed into three regions $M_\beta^-,M_\beta^+$ and $M_\beta$}
		\end{subfigure}
		\hfill
		\begin{subfigure}[t]{0.45\linewidth}
			\centering\raisebox{\dimexpr.5\ht\imageboxfour-.5\height}{
				
				\begin{tikzpicture}
					
					\draw[thick,fill=gray!45](0,2.5)--(6,2.5)--(6,4.5)--(0,4.5)--(0,2.5);
					\node at (3,3.5) {$M_\beta$} ;
					
					\begin{scope}[ thick,decoration={
							markings,
							mark=at position 0.5 with {\arrow{>>}}}
						] 
						\draw[postaction={decorate}] (6,2.5)--(6,4.5);
						\draw[postaction={decorate}] (0,4.5)--(0,2.5);
					\end{scope}

					\draw[thick,fill=gray!15](0,1)--(6,1)--(6,0)--(0,0)--(0,1);
					\node at (3,.5) {$M_\beta^-$} ;
					
					\draw[thick,fill=gray!15](0,2)--(6,2)--(6,1)--(0,1)--(0,2);
					\node at (3,1.5) {$M_\beta^+$} ;
					
					\begin{scope}[ thick,decoration={
							markings,
							mark=at position 0.6 with {\arrow{>}}}
						] 
						\draw[postaction={decorate}] (6,0)--(6,1);
						\draw[postaction={decorate}] (0,2)--(0,1);
					\end{scope}
					
					\begin{scope}[ thick,decoration={
							markings,
							mark=at position 0.6 with {\arrow{>>>}}}
						] 
						\draw[postaction={decorate}] (6,1)--(6,2);
						\draw[postaction={decorate}] (0,1)--(0,0);
					\end{scope}
					
					\begin{scope}[ thick,decoration={
							markings,
							mark=at position 0.55 with {\arrow{>>>>}}}
						] 
						\draw[postaction={decorate}] (0,1)--(6,1);
					\end{scope}
					
				\end{tikzpicture}
			}
			\caption*{$M_\beta$ is a Möbius strip of width $2\beta$, $M_\beta^-$ and $M_\beta^+$ can be glued together to form a Möbius strip of the same width}
		\end{subfigure}
		\hfill

		\caption{The Klein bottle $K_\beta$ contains two Möbius strips of width $2\beta$ }
		\label{Fig: Klein Bottle from Mobius}
	\end{figure}
	
	The maps 
	\begin{align}
		\hat I &:\R^2 \to \R^2, &  \hat I(x,y)&=(x,-y)\\
		\hat J &:\R^2 \to \R^2,  & 	 \hat J(x,y)&=(x,y+2\beta)\\
		\hat T_h &:\R^2 \to \R^2, & 	 \hat T_h(x,y)&=(x+h,y), \quad \textnormal{for }h\in \R
	\end{align}
	are isometries of $(\R^2,g_{flat})$ and commute with $\Gamma_\beta$, so they induce isometries $I, J, T_h\in H_\beta:=\Isom(K_\beta, g_{flat})$. 
	
	By \Cref{Lemma: Characterization of CK}, every Riemannian metric on a Klein bottle is isometric to a metric of the form $\phi^2g_{flat}$ on $K_\beta$, for some $\beta>0$ and some smooth positive function $\phi : K_\beta\to (0,+\infty)$. Similar to the situation on the Möbius strip, $\Gamma_\beta$-invariant functions on $\R^2$ induce functions on $K_\beta$ and vice versa, so we use the two notions interchangeably. It follows from \Cref{Lemma: isometry projection} that metrics conformally equivalent to $g_{flat}^{C_\beta}$ which minimize the systolic area  are necessarily $H_\beta$-invariant, i.e. the conformal factor $\phi$ has to be constant in $x$ and satisfy
	\begin{align}\label{Eq: invariance properties of conformal factor on Kbeta}
		\phi(x,y)=\phi(x,-y)=\phi(x,y+2\beta).
	\end{align}
	This shows that all information about $\phi$ can be retrieved from the function $[0,\beta]\to (0,+\infty) $, $y\mapsto \phi(x,y)$, which we also denote by $\phi$. Conversely, any function on $[0,\beta]$ can be extended to a function on $K_\beta$ via \eqref{Eq: invariance properties of conformal factor on Kbeta}. Note that this extension process does yield continuous, but not  necessarily smooth functions, even if the initial function was smooth on $[0,\beta]$. Another consequence of \eqref{Eq: invariance properties of conformal factor on Kbeta} is that the $L^2$-inner product of $H_\beta$-invariant functions $f,h\in L^2(K_\beta,g_{flat})$ is given by
	\begin{equation}\label{Eq: L2 product on Klein}
		(f,h)_{L^2}=\int_{K_\beta} f(x,y)h(x,y)dxdy=4\pi \int_0^\beta f(y)h(y)dy,
	\end{equation}
	similar to \eqref{Eq: L2 product on Möbius} on the Möbius strip.
	
	In \cite{Bavard1988}, Bavard constructed (continuous, not necessarily smooth) metrics $g_\beta=\phi^2_\beta g_{flat}^{C_\beta}$ of minimal systolic area in each conformal class $C_\beta$ with the normalization $\sys(K_\beta, g_\beta)=\pi$. He proved the minimality using a criterion from \cite{Bavard1992}. Projecting onto the  conformal factor $\phi_\beta$ (similar to the projection in \Cref{Thm: Systolic inequality with remainder for the Mobius stip}) provides an alternative proof of minimality while also providing an estimate for the systolic defect. 
	
	The construction of $\phi_\beta$ is mostly based on the function $\phi_0$ (defined in  \eqref{Eq: Definition of phi_0}) and divides into four cases. In each case, \eqref{Eq: invariance properties of conformal factor on Kbeta} is used to obtain functions defined on $K_\beta$ from the description of  $\phi_\beta$ on $[0,\beta]$. 
	We denote the metric defined via $\phi_\beta$ by $g_{\beta}=\phi_\beta^2g_{flat}^{C_\beta}$.
	\begin{Case}\label{Case: flat bottle}
		$\beta\leq \frac{\pi}{4}$. 
		Here, $\phi_\beta$ is given by 
		\begin{equation}
			\phi_\beta : [0,\beta]\to(0,+\infty),\quad \phi_\beta(y)=\frac{\pi}{4\beta}.
		\end{equation}
		$(K_\beta,g_\beta)$ satisfies $\area(K_\beta,g_\beta)=\frac{\pi^3}{4\beta}$.
	\end{Case}
	\begin{Case}\label{Case: round-flat thin bottle}
		$\frac{\pi}{4}<\beta <\beta_0:=\log(1+\sqrt[]{2})$. 
		Here, $\phi_\beta$ is given by 
		\begin{equation}
			\phi_\beta : [0,\beta]\to(0,+\infty),\quad \phi_\beta(y)=
			\begin{cases}
				\phi_0(y), \quad &\textnormal{for } y \leq s_\beta,\\
				\phi_0(s_\beta), &\textnormal{for } y > s_\beta,
			\end{cases}
		\end{equation}
		where $s=s_\beta$ is the unique solution of the equation
		\begin{equation}
			\int_0^s \phi_0(y)dy + (\beta-s)\phi_0(s) = \pi/4, \quad s\in (0,\beta).
		\end{equation}
		$(K_\beta,g_\beta)$ satisfies $\area(K_\beta,g_\beta)=4\pi \tanh(s_\beta)+4\pi (\beta-s_\beta)\phi_0^2(s_\beta)$. For the definition of $s_\beta$, compare also \Cref{Fig: Graph of phibeta Klein bottle}.
	\end{Case}
	
	\begin{Case}\label{Case: round bottle}
		$\beta_0\leq \beta \leq \beta_1:=\log(2+\sqrt[]{3})$. 
		Here, $\phi_\beta$ is given by 
		\begin{equation}
			\phi_\beta : [0,\beta]\to(0,+\infty),\quad \phi_\beta(y)=\phi_0(y).
		\end{equation}
		$(K_\beta,g_\beta)$ satisfies $\area(K_\beta,g_\beta)=4\pi \tanh (\beta)$.
	\end{Case}
	
	\begin{Case}\label{Case: round-flat thick bottle}
		$\beta > \beta_1$. 
		Here, $\phi_\beta$ is given by 
		\begin{equation}
			\phi_\beta : [0,\beta]\to(0,+\infty),\quad \phi_\beta(y)=
			\begin{cases}
				\phi_0(y), \quad &\textnormal{for } y \leq \beta_1,\\
				\frac{1}{2}=\phi_0(\beta_1), &\textnormal{for } y > \beta_1,
			\end{cases}
		\end{equation}
		$(K_\beta,g_\beta)$ satisfies $\area(K_\beta,g_\beta)=4\pi \tanh(\beta_1)+\pi (\beta-\beta_1)=2\pi\, \sqrt[]{3}+\pi(\beta-\beta_1)$.
	\end{Case}
	\setcounter{Case}{0}

	\begin{figure}
		\centering
		
		\begin{tikzpicture}[domain=-.5:1.762,scale=3]

			\draw[->] (-.5,0) -- (2.3,0) node[right] {$y$};
			\draw[->] (0,-.3) -- (0,1.3);

			\draw[dashed, thin] (-.5,1)--(2.3,1);
			\node[anchor= east] at  (-.5,1){$1$};	
			
			\draw[dashed, thin] (1.762,-.1) -- (1.762,1.3);
			\node[anchor=north] at  (1.762,-.1){$\beta$};	
			
			\draw[dashed, thin] (1,-.1) -- (1,0.887);
			\node[anchor=north] at  (1,-.1){$s_\beta$};

			\node[anchor=west] at (1.762,.887) {$\phi_\beta$};
			\node[anchor=south west] at (1.3,.58) {$\phi_0$};			
			
			\draw[thick]    plot[domain=-.5:1] (\x,{2*exp(\x/2)/(1+exp(\x))}  ) ;
			\draw[thick]    plot[domain=1:1.762] (\x,{2*exp(1/2)/(1+exp(1))}  ) ;
			\draw[dotted]    plot[domain=1:1.762] (\x,{2*exp(\x/2)/(1+exp(\x))}  ) ;
			
		\end{tikzpicture}

		\caption{The graph of $\phi_\beta$ for $\frac{\pi}{4}<\beta<\beta_0$. The function $s\mapsto I(s):=\int_0^s \phi_0(y)dy + (\beta-s)\phi_0(s) $ gives the integral of the continuous function that restricts to $\phi_0$ on $[0,s]$ and is constant on $[s,\beta]$. $I$ is continuous, strictly decreasing on $[0,\beta]$ and satisfies $I(0)=\beta>\pi/4$, $I(\beta)=\int_0^\beta \phi_0(y)dy< \pi/4$. $s_\beta:=I^{-1}(\pi/4)\in (0,\beta)$ is hence well-defined.}
		\label{Fig: Graph of phibeta Klein bottle}
	\end{figure}
	
	The fact that $g_\beta$ satisfies $\sys(K_\beta,g_\beta)=\pi$ for all $\beta>0$ can be seen as follows: 
	consider a noncontractible closed curve $\gamma:[a,b]\to K_\beta$ and a lift $\hat{\gamma}=(\hat{x},\hat{y}):[a,b]\to \R^2$ of $\gamma$ to the universal cover. We call $\gamma$ \textit{vertical}, if the image of $\hat{y}$ contains an interval of length $4\beta$ (or equivalently, if $\gamma$ considered as a map to the fundamental region $[0,\pi)\times [-2\beta,2\beta)$ intersects the line $[0,\pi)\times\{c\}$ nontrivially for every $c\in [-2\beta,2\beta)$). Otherwise, $\gamma$ is called \textit{horizontal}. 
	If $\gamma$ is vertical, the invariance in the $x$-direction of $\phi_\beta$ and consequently of $g_\beta$ implies $L_{g_\beta}(\gamma)\geq L_{g_\beta}(\gamma_0)$, where $\gamma_0$ is given by
	\begin{equation}
		\gamma_0: [-2\beta,2\beta]\to K_\beta,\quad \gamma_0(t)=p(0,t). 
	\end{equation}
	A direct calculation involving the invariance properties of $\phi_\beta$ shows 
	\begin{equation}
		L_{g_\beta}(\gamma_0)=\int_{-2\beta}^{2\beta} \phi_\beta(t)dt=4\int_0^\beta \phi_\beta(t)dt. 
	\end{equation}
	By construction, we have $\int_0^\beta \phi_\beta(y)dy\geq \pi/4$, hence $L_{g_\beta}(\gamma)\geq \pi$ holds for every vertical curve $\gamma$.
	
	If $\gamma$ is horizontal, we can replace it by a curve $\tilde{\gamma}$ of the same length which is completely contained in the Möbius strip $M_\beta\subset K_\beta$. For this, consider the map 
	\begin{equation*}
		[0,\pi]\times [-2\beta,2\beta]\to [0,\pi]\times [-\beta,\beta],\quad (x,y)\mapsto 
		\begin{cases}
			(x,y)\quad 		&\textnormal{for } y\in [-\beta,\beta]\\
			(x,2\beta- y) 	&\textnormal{for } y\in (\beta,2\beta]\\
			(x,-2\beta- y) 	&\textnormal{for } y\in [-2\beta,-\beta).
		\end{cases}
	\end{equation*}
	It induces a continuous map $R:K_\beta\to M_\beta$ which restricts to the identity on $M_\beta$ and maps the region $M_{\beta}^\mp$ to its reflection along the line $p(\{(x,\pm \beta)|x\in [0,\pi]\})$ (compare \Cref{Fig: Klein Bottle from Mobius}). Now define $\tilde{\gamma}:=R\circ\gamma$. 
	The fact that $\gamma$ is horizontal implies that its lift $\hat{\gamma}=(\hat{x},\hat{y})$ satisfies $\hat{x}(a)\neq \hat{x}(b)$, hence $\tilde{\gamma}:[a,b]\to M_\beta$ is noncontractible. 
	A direct comparison between $\phi_\beta$ and the conformal factor of the optimal metric $g_{\beta}^M$ in the conformal class $C_\beta\in \mathpcal{C}(M)$ on the Möbius strip $M_\beta$ (see \eqref{Eq: Def of conf factor mobius}) gives $L_{g_\beta}(\gamma)=L_{g_\beta}(\tilde{\gamma})\geq L_{g^M_\beta}(\tilde{\gamma})$. 
	Combined with $L_{g^M_\beta}(\tilde{\gamma})\geq\sys(M_\beta,g_{\beta}^M)=\pi$, we obtain $L_{g_\beta}(\gamma)\geq \pi$ for every horizontal curve $\gamma$, thus proving $\sys(K_\beta,g_\beta)\geq \pi$. 
	
	For the other inequality, one can simply consider curves attaining the systole: in \Cref{Case: flat bottle} (and \Cref{Case: round-flat thin bottle}), the vertical curve $\gamma_0$ is of length $\pi$, in \Cref{Case: round-flat thin bottle,Case: round bottle,Case: round-flat thick bottle} the horizontal curve $\gamma:[0,\pi]\to K_\beta$, $\gamma(t)=p(t,0)$ has length $\pi$.
	
	Similar to the notion of horizontal and vertical curves presented here, Definition 1.1 of \cite{ElMir2015} distinguishes certain homotopy classes of noncontractible closed curves on the Klein bottle. The minimal lengths of curves in these homotopy classes are then related to the area of the Klein bottle and yield inequalities similar to the conformal inequalities found in \cite{Blatter1961a}. However, will not discuss these results here.

	\begin{Rem}\label{Rem: motivation for construction}
		The constructions for all four cases have geometric interpretation when one considers the embedded Möbius strip $M_\beta\subset K_\beta$: the normalization $\sys(K_\beta,g)=\pi$ implies $\sys(M_\beta,\restr{g}{M_\beta})\geq \pi$ and $(K_\beta,g)$ has precisely twice as much area as $(M_\beta,g)$. So in order to construct a systolically optimal metric $g=\phi^2g_{flat}$ on $K_\beta$, one tries to minimize the area of $(M_\beta,g)$ (or equivalently, the $L^2$-norm of $\phi$) under the constraints $\sys(M_\beta,g)\geq \pi$ as well as $\int_0^\beta \phi(y)dy\geq \pi/4$. The latter inequality comes from the fact that the vertical curves also have to be of length at least $\pi$. 
		
		In the thin case (\Cref{Case: flat bottle}), these vertical curves (described precisely in \Cref{Sec: Projection Lemma}) completely dominate the situation (in the sense that if one minimizes the $L^2$-norm of $\phi$ while satisfying $\int_0^\beta \phi(y)dy\geq \pi/4$, the constraint $\sys(M_\beta,\restr{g}{M_\beta})\geq \pi$ is automatically satisfied), so the fact that the flat Klein bottle (i.e. $\phi$ being constant) is area-minimizing follows directly from the Cauchy-Schwarz-inequality. 
		
		In \Cref{Case: round-flat thin bottle}, both the vertical curves as well as the horizontal closed curves in $M_\beta$ play an important role. One would like to use just $\phi_0$, as it is area-minimizing with the property $\sys(M_\beta,g)\geq \pi$, but for $\beta<\beta_0$ this does not satisfy $\int_0^\beta \phi(y)dy\geq \pi/4$ (i.e. the vertical curves are too short). It turns out (this is what makes it a bit more complicated) that the two constraints can be satisfied in an area-minimizing way by taking $\phi_\beta(y)=\phi_0(y)$ for $y\leq s_\beta$ and constant for $y>s_\beta$, where the definition of $s_\beta$ precisely ensures $\int_0^\beta \phi_\beta(y)dy=\pi/4$. 
		
		For \Cref{Case: round bottle,Case: round-flat thick bottle}, the embedded Möbius strip is wide enough so that the vertical curves are naturally long, (i.e. $\sys(M_\beta,\restr{g}{M_\beta})\geq \pi$ already implies $\int_0^\beta \phi(y)dy\geq \pi/4$), hence one only needs to ensure that all horizontal curves have length at least $\pi$. Consequently, the construction is similar to the construction on the Möbius strip. 
	\end{Rem}
	For all $\beta>0$ we have $\sys(M_\beta,g_{\beta})=\pi$, so the systolic area is given by
	\begin{equation}
		\alpha_{sys}(K_\beta,g_\beta)=
		\begin{cases}
			\frac{\pi}{4\beta},\quad &\textnormal{for } \beta\leq \frac{\pi}{4},\\			
			\frac{4}{\pi} \tanh(s_\beta)+\frac{4}{\pi}(\beta-s_\beta)\phi_0^2(s_\beta),\quad &\textnormal{for } \frac{\pi}{4}<\beta<\beta_0,\\
			\frac{4}{\pi}\tanh \beta,\quad &\textnormal{for }\beta_0\leq \beta\leq \beta_1 ,\\
			\frac{2\,\sqrt[]{3}}{\pi} +\frac{\beta-\beta_1}{\pi} ,\quad &\textnormal{for } \beta > \beta_1.
		\end{cases}
	\end{equation}
	Note that $\alpha_{sys}(K_{\beta_0},g_{\beta_0})=\frac{2\, \sqrt[]{2}}{\pi}$ is the minimum of the function $\beta \mapsto \alpha_{sys}(K_\beta,g_\beta)$, with $\beta_0=\log(1+\,\sqrt[]{2})$ being the unique minimizer (see also \Cref{Fig: Graph of C mapsto alpha_sys} for a sketch of the graph). This fact is used to retrieve the systolic inequality for metrics in all conformal classes on the Klein bottle due to Bavard (see \cite{Bavard1986}) and provides a term depending on the conformal class in the estimate on the systolic defect (compare \Cref{Cor: Bavards inequality with defect}).
	
	Analogously to the situation on the Möbius strip, the map 
	\begin{equation}
		P_{\phi_\beta}: L^2(K_\beta, g_{flat}) \to L^2(K_\beta, g_{flat}),\quad P_{\phi_\beta}(\phi)=\frac{(\phi,\phi_\beta)_{L^2}}{(\phi_\beta,\phi_\beta)_{L^2}}\phi_\beta.
	\end{equation} 
	is the $L^2$-orthogonal projection onto the subspace $\R\phi_\beta  \subset L^2(K_\beta, g_{flat})$ and has essentially the same properties as the projection on the Möbius strip:
	\begin{Lemma}\label{Lemma: Projection Lemma for Klein bottle}
		Let $g=\phi^2 g_{flat}^{C_\beta}$ be a $H_\beta$-invariant Riemannian metric on $K_\beta$. Then the metric $P_{\phi_\beta}(\phi)^2 g_{flat}$ satisfies
		\begin{align}
			\area(K_\beta,g) &= \area (K_\beta, P_{\phi_\beta}(\phi)^2 g_{flat}) + \|\phi-P_{\phi_\beta}(\phi)\|_{L^2}^2 \label{Eq: Equality in klein-projection}\\
			\sys(K_\beta,g)&\leq \sys (K_\beta, P_{\phi_\beta}(\phi)^2 g_{flat})			\label{Eq: projection-ineq for klein proj}
		\end{align}
	\end{Lemma}
	Equality in \eqref{Eq: Equality in klein-projection} can be proven analogous to the proof of the similar equality in  \Cref{Lemma: isometry projection}.
	However, proving that $P_{\phi_\beta}$ does not decrease the systole requires more than just Pu's inequality \eqref{Eq: Pus inequality}, because in $K_\beta$ there are more noncontractible closed curves to consider than in $M_\beta$. It is a consequence of the following result, whose proof can be found in \Cref{Sec: Projection Lemma}:
	\begin{Lemma}[Projection inequality]\label{Lemma: Projection Lemma}
		Let $\beta>0$ and $g=\phi^2g_{flat}^{C_\beta}$ be a $H_\beta$-invariant Riemannian metric on $K_\beta$ such that $\sys(M_\beta,g)=\pi$. Then 
		\begin{equation}
			\int_0^\beta \phi(y)\phi_\beta(y)dy\geq \int_0^\beta \phi_\beta^2(y)dy. 
		\end{equation} 
	\end{Lemma}
	We now show how \eqref{Eq: projection-ineq for klein proj} follows from the projection inequality, thus concluding the proof of \Cref{Lemma: Projection Lemma for Klein bottle}.
	Due to the scaling-invariance of \eqref{Eq: projection-ineq for klein proj}, we can assume that $\sys(K_\beta,g)=\pi$. Then it follows from a combination of \eqref{Eq: L2 product on Klein} and \Cref{Lemma: Projection Lemma} that 
	\begin{equation}
		\frac{(\phi,\phi_\beta)_{L^2}}{(\phi_\beta,\phi_\beta)_{L^2}}\geq 1
	\end{equation}
	and hence $P_{\phi_\beta}(\phi)\geq \phi_\beta$. Thus
	\begin{equation}
		\sys(K_\beta,P_{\phi_\beta}(\phi)^2g_{flat})\geq \sys(K_\beta, \phi_\beta^2 g_{flat})=\pi = \sys(K_\beta,g), 
	\end{equation}
	completing the proof. 
	
	As a consequence, we obtain the conformal systolic inequality of \cite{Bavard1988} with an estimate on the systolic defect, which also implies the general systolic inequality on the Klein bottle from \cite{Bavard1986}. 
	\begin{Thm}\label{Thm: conformal inequality klein with defect}
		Let $\beta>0$ and $g=\phi^2 g_{flat}^{C_\beta}\in C_\beta$ be a Riemannian metric on $K_\beta$. 
		Then 
		\begin{equation}
			\alpha_{sys}(K,C_\beta)=\alpha_{sys}(K_\beta,g_\beta),
		\end{equation}
		and
		\begin{equation}\label{Eq: conformal inequality klein with defect}
			\dfrac{\|\phi- P_{\phi_\beta}(\phi)\|^2_{L^2}}{\sys^2(K_\beta,g)}\leq \alpha_{sys}(K_\beta,g)-\alpha_{sys}(K,C_\beta).
		\end{equation}
		In particular, $\alpha_{sys}(K_\beta,g)=\alpha_{sys}(K_\beta,C_\beta)$ holds if and only if $g$ is a constant multiple of $g_\beta=\phi_\beta^2g_{flat}$. 
	\end{Thm}
	The proof is a direct consequence of \Cref{Lemma: isometry projection,Lemma: Projection Lemma for Klein bottle} and similar to the proof of \Cref{Thm: Systolic inequality with remainder for the Mobius stip}. Using the fact that $\beta\mapsto \alpha_{sys}(K,C_\beta)$ attains its minimum of $\frac{2\sqrt{2}}{\pi}$ at the unique minimizer $\beta=\beta_0$ gives the following stability estimate for Bavard's original systolic inequality for the Klein bottle:
	\begin{Cor}\label{Cor: Bavards inequality with defect}
		Let $g=\phi^2g_{flat}\in C_\beta$ be a Riemannian metric on a Klein bottle $K_\beta$, for some $\beta>0$. 
		Then 
		\begin{equation}
			\alpha_{sys}(K)=\alpha_{sys}(K_{\beta_0},g_{\beta_0})=\frac{2\sqrt{2}}{\pi}
		\end{equation}
		and
		\begin{align}
			\alpha_{sys}(K,C_\beta)-\alpha_{sys}(K)	+ \dfrac{\|\phi- P_{\phi_\beta}(\phi)\|^2_{L^2}}{\sys^2(K_\beta,g)} 
			\leq \alpha_{sys}(K_\beta,g)-\frac{2\sqrt{2}}{\pi}.
		\end{align}
		In particular, $\alpha_{sys}(K_\beta,g)=\frac{2\sqrt{2}}{\pi}$ holds if and only if $\beta=\beta_0$ and if $g$ is a constant multiple of $g_{\beta_0}$. 
	\end{Cor} 
	
	\begin{Rem}\label{Rem: Klein bottle variance}
		Similar to the discussion preceding \Cref{Cor: Variance formulation Mobius}, any metric $g\in C_\beta$ can be written as $g=h^2g_\beta$ and $g=\phi^2g_{flat}^{C_\beta}$. The measure induced by $g_\beta$ can be normalized by the factor $\frac{1}{\area(K_\beta,g_\beta)}$ to give a probability measure, so  expected value $E(h)$ and variance $\Var(h)$ can be defined with respect to it. Direct calculations show that $P_{\phi_\beta}(\phi)=E(h)\phi_\beta$, hence \eqref{Eq: conformal inequality klein with defect} reformulates to 
		\begin{equation}
			\area(K_\beta,g_\beta) \dfrac{\Var(h)}{\sys^2(K_\beta,g)}\leq \alpha_{sys}(K_\beta,g)-\alpha_{sys}(K,C_\beta).
		\end{equation}
	\end{Rem}
	
	\section{Proof of the Projection Inequality}\label{Sec: Projection Lemma}
	
	The proof of \Cref{Lemma: Projection Lemma} divides into the four cases from the definition of $\phi_\beta$. \Cref{Case: flat bottle,Case: round bottle,Case: round-flat thick bottle} have fairly short proofs, however, for \Cref{Case: round-flat thin bottle} the situation is more complicated. Recall that $p:\R^2 \to K_\beta$ denotes the quotient map. 
	\begin{proof}[Proof of \Cref{Case: flat bottle}]
		We have $\beta\leq \frac{\pi}{4}$ and $\sys(K_\beta,g)=\pi$. The vertical curve $\gamma: [-2\beta,2\beta]\to K_\beta$, $\gamma(t)=p(0,t)$ is closed, noncontractible and its length is given by
		\begin{equation}\label{Eq: length of vertical curve}
			L_g(\gamma)=\int_{-2\beta}^{2\beta} \sqrt[]{g_{\gamma(t)}(\gamma'(t),\gamma'(t))} dt =  \int_{-2\beta}^{2\beta} \phi(t) dt = 4\int_0^\beta \phi(t)dt.
		\end{equation}
		By the definition of the systole, we have $L_g(\gamma)\geq \sys(K_\beta,g)=\pi$, so \eqref{Eq: length of vertical curve} gives us $\int_0^\beta \phi(t)dt\geq \frac{\pi}{4}$. As $\phi_\beta$ is constant, we can simply calculate
		\begin{equation}
			\int_0^\beta \phi(y)\phi_\beta(y) dy = \frac{\pi}{4\beta}\int_0^\beta \phi(y)dy \geq \frac{\pi}{4\beta}\frac{\pi}{4}=\frac{1}{\beta} \left(\frac{\pi}{4}\right)^2
		\end{equation} 
		and 
		\begin{equation}
			\int_0^\beta \phi_\beta^2(y) dy = \beta \left(\frac{\pi}{4\beta}\right)^2= \frac{1}{\beta} \left(\frac{\pi}{4}\right)^2,
		\end{equation}
		which gives the desired inequality
	\end{proof}
	\begin{proof}[Proof of \Cref{Case: round bottle}]
		Here, we are in the situation $\beta_0\leq \beta\leq \beta_1$, and $\phi_\beta=\phi_0$. 
		The Möbius strip $(M_\beta,\restr{g}{M_\beta})$ contained in $(K_\beta,g)$ satisfies $\sys(M_\beta, \restr{g}{M_\beta})\geq \sys(K_\beta,g)=\pi$ and $\restr{g}{M_\beta}=\restr{\phi^2}{M_\beta}g_{flat}$ is invariant under the isometry group of $(M_\beta,g_{flat})$ (due to the $H_\beta$-invariance of $g$). Thus, the assumptions for Pu's inequality \eqref{Eq: Pus inequality} are fulfilled and we obtain
		\begin{equation}
			\int_0^\beta \phi(y) \phi_0(y) dy  \geq \int_0^\beta \phi_0^2(y) dy. 
		\end{equation}
		Due to $\phi_\beta=\phi_0$, this finishes the proof.  
	\end{proof}
	\begin{proof}[Proof of \Cref{Case: round-flat thick bottle}]
		We now consider $\beta>\beta_1$. As on the thick Möbius strip, 
		the horizontal curves that wrap around twice, given by $\lambda_\tau:[-\pi/2,3\pi/2]\to K_\beta$, $\lambda_\tau(t)=p(t,\tau)$, have to be of length (with respect to the metric $g$) at least $\pi=\sys(K_\beta,g)$, which implies $\phi\geq \frac{1}{2}$. By restricting to the Möbius strip $M_{\beta_1}\subset K_\beta$ (where $\phi_\beta=\phi_0$ holds) and proceeding as above, Pu's inequality yields 
		\begin{equation}
			\int_0^{\beta_1}\phi(y)\phi_\beta(y)dy\geq \int_0^{\beta_1}\phi_\beta^2(y)dy. 
		\end{equation}
		Combining this with $\phi_0\geq \frac{1}{2}=\phi_\beta(y)$ for $y\geq \beta_1$ gives the desired
		\begin{align*}
			\int_0^\beta \phi(y)\phi_\beta(y)dy&=\int_0^{\beta_1} \phi(y)\phi_\beta(y)dy+\int_{\beta_1}^\beta \phi(y)\phi_\beta(y)dy\\
			&\geq \int_0^{\beta_1}\phi_\beta^2(y)dy + \int_{\beta_1}^\beta \phi_\beta(y)\phi_\beta(y)dy\\
			&=\int_0^\beta \phi_\beta^2(y)dy. 
		\end{align*}
	\end{proof}
	The proof of the remaining \Cref{Case: round-flat thin bottle} is based on the following general Lemma:
	\begin{Lemma}\label{Lemma: Decreasing function Lemma}
		Let $s>0$, $h:[0,s]\to (0,+\infty)$ be a continuous, positive and decreasing function and $g:[0,s]\to \R$ be continuous such that
		\begin{equation}
			\int_0^z h(y)g(y)dy \geq 0, \quad \forall z\in [0,s].
		\end{equation}
		Then 
		\begin{equation}
			\int_0^z h(y)g(y)dy \geq h(z)\int_0^z g(y)dy, \quad \forall z \in [0,s].
		\end{equation}
	\end{Lemma}
	\begin{proof}
		Consider the continuous function $u:[0,s]\to \R$, $u(z)=\int_0^z h(y)g(y)dy - h(z)\int_0^z g(y)dy$. The claimed inequality is equivalent to $u(z)\geq 0$ for all $z\in [0,s]$. To prove this, we argue by contradiction: assume there exists $z_0\in [0,s]$ such that $u(z_0)<0$. Because of $u(0)=0$, we have $z_0>0$. The continuity of $u$ implies that $a:=\sup\{z\in [0,z_0] |u(z)\geq 0 \}$ is in $[0,z_0)$ and satisfies $u(a)=0$ as well as $u(z)<0$ for all $z\in (a,z_0]$. 
		
		We now recall the definition of the (lower right-handed) Dini derivative. A good reference on this topic is Chapter 3 of \cite{Kannan1996}. Let $f:[a,b]\to \R$ be a function and $x\in [a,b)$. Then 
		\begin{equation}
			D_+ f(x):= \liminf_{t\to 0+} \frac{f(x+t)-f(x)}{t}
		\end{equation}
		is called the lower right-handed Dini derivative of $f$ at $x$. Similar definitions for the upper and left-handed Dini derivatives also exist. 
		A calculation shows that in our case we have
		\begin{equation}\label{Eq: Dini of u }
			D_+u(z)=-D_+h(z)\int_0^z g(y)dy 
		\end{equation}
		for $z\in [0,s)$. As $h$ is a decreasing function, we have $\frac{h(z+t)-h(z)}{t}\leq 0$ for all $t\in (0,s-z)$ and thus $D_+h(z)\leq 0$. For $z\in (a,z_0]$, we have $u(z)<0$, which is equivalent to 
		\begin{equation}
			\int_0^z h(y)g(y)dy<h(z)\int_0^z g(y)dy.
		\end{equation}
		The assumptions $h>0$ and $\int_0^z h(y)g(y)dy\geq 0$ for all $z\in [0,s]$ therefore imply 
		\begin{equation}
			\int_0^z g(y)dy>0
		\end{equation}
		for all $z\in (a,z_0]$. Combined with \eqref{Eq: Dini of u } and $D_+h\leq 0$, this yields 
		\begin{equation}
			D_+u(z)\geq 0 
		\end{equation}
		for all $z\in (a,z_0)$. Hence, $\restr{u}{[a,z_0]}$ is a continuous function with nonnegative Dini derivative $D_+u$ on $(a,z_0)$. This implies that $\restr{u}{[a,z_0]}$ is non-decreasing (see for example \cite{Kannan1996}, Theorem 3.4.4). Consequently, we have 
		\begin{equation}
			0=u(a)\leq u(z_0)<0,
		\end{equation}
		which is a contradiction. Thus, such $z_0$ cannot exist, proving $u(z)\geq 0$ for all $z\in [0,s]$.
	\end{proof}
	\begin{Rem}
		If we assume the function $h$ in the above Lemma to be differentiable, then so is the function $u$ used in the proof. In this case, one can compute the (usual) derivative of $u$  to show that it is increasing on $[a,z]$, simplifying the proof. 
	\end{Rem}
	
	\begin{proof}[Proof of \Cref{Case: round-flat thin bottle}]
		Here we are in the situation $\beta\in (\beta_0,\beta_1)$, $\phi_\beta$ is given by 
		\begin{equation}
			\phi_\beta(y)=	\begin{cases}
				\phi_0(y)\quad 		&\textnormal{for } \beta \leq s_\beta,\\
				\phi_0(s_\beta)	&\textnormal{for } \beta > s_\beta,
			\end{cases}	
		\end{equation}
		where the definition of $s_\beta$ ensures precisely $\int_0^\beta \phi_\beta(y)dy=\frac{\pi}{4}$. For   $z\leq s_\beta$, the restriction to the Möbius strip $M_z\subset K_\beta$ combined with Pu's inequality, similar to the arguments in the proofs of \Cref{Case: round-flat thick bottle,Case: round bottle}, gives us 
		\begin{equation}\label{Eq: Consequence of Pu for case 3}
			\int_0^z \phi(y)\phi_\beta(y) dy \geq \int_0^z \phi_\beta^2(y)dy. 
		\end{equation}
		Similar to the situation in \Cref{Case: flat bottle}, the vertical curve $\gamma$ ensures $\int_0^\beta \phi(y)dy\geq \frac{\pi}{4}$. 
		We write $\phi=\phi_{\beta}+\sigma$, for $\sigma:[0,\beta]\to \R$. \eqref{Eq: Consequence of Pu for case 3} thus reformulates to 
		\begin{equation}
			\int_0^z \phi_\beta^2(y) dy \leq \int_0^z (\phi(y)+\sigma(y))\phi_\beta(y)dy=\int_0^z \phi_{\beta}^2(y)dy+\int_0^z \sigma(y)\phi_{\beta}(y)dy,
		\end{equation}
		which implies 
		\begin{equation}
			\int_0^z \sigma(y)\phi_{\beta}(y)dy \geq 0, \quad \forall z\leq s_\beta.
		\end{equation}
		Hence, $h=\phi_\beta$ and $g=\sigma$ satisfy the assumptions of \Cref{Lemma: Decreasing function Lemma} on the interval $[0,s_\beta]$. Consequently, we find
		\begin{equation}\label{Eq: Result of the decr funct lemma}
			\int_0^{s_\beta} \sigma(y)\phi_{\beta}(y)dy\geq \phi_\beta(s_\beta)\int_0^{s_\beta}\sigma(y)dy. 
		\end{equation}
		The inequality $\int_0^\beta \phi(y)dy\geq \pi/4$ gives us
		\begin{equation}
			\pi/4\leq \int_0^\beta (\phi_\beta(y)+\sigma(y))dy=\int_0^\beta \phi_\beta(y)dy+\int_0^\beta \sigma(y)dy,
		\end{equation}
		which, in combination with $\int_0^\beta \phi_\beta(y)dy=\pi/4$, yields precisely $\int_0^\beta \sigma(y)dy\geq 0$. 
		Using that and \eqref{Eq: Result of the decr funct lemma}, we calculate
		\begin{align*}
			\int_0^\beta \phi(y)\phi_\beta(y)dy &= \int_0^\beta (\phi_\beta(y)+\sigma(y))\phi_\beta(y)dy = \int_0^\beta \phi_\beta^2(y)dy+\int_0^\beta \sigma(y)\phi_\beta(y)dy \\
			&= \int_0^\beta \phi_{\beta}^2(y)dy+ \int_0^ {s_\beta}\sigma(y)\phi_{\beta}(y)dy + \int_{s_\beta}^\beta \sigma(y)\underbrace{\phi_\beta(y)}_{=\phi_\beta(s_\beta)}dy \\
			&\geq  \int_0^\beta \phi_{\beta}^2(y)dy+ \phi_\beta(s_{\beta})\int_0^ {s_\beta}\sigma(y)dy + \phi_\beta(s_\beta) \int_{s_\beta}^\beta \sigma(y)dy \\
			&= \int_0^\beta \phi_{\beta}^2(y)dy+ \phi_\beta(s_\beta) \left(\int_0^ {s_\beta}\sigma(y)dy + \int_{s_\beta}^\beta \sigma(y)dy\right) \\
			&= \int_0^\beta \phi_{\beta}^2(y)dy+ \underbrace{\phi_\beta(s_\beta)}_{>0} \underbrace{\int_0^ {\beta}\sigma(y)dy}_{\geq 0 } \\
			&\geq \int_0^\beta \phi_{\beta}^2(y)dy,
		\end{align*}
		concluding the proof.
	\end{proof}

	\printbibliography
\end{document}